\documentclass[11pt]{article}					
\usepackage{amsthm}     
\usepackage[english]{babel}                 
\usepackage{amsmath, amssymb, amsfonts,hyperref}     
\usepackage{dsfont}
\usepackage{graphicx}                       
\usepackage{color,bbm}                          
\usepackage[normalem]{ulem}                           
\usepackage{tikz}
\usepackage{soul}
\usepackage{subfigure}
\graphicspath{{Images/}}
\newcommand{\A}{\mathcal{A}}
\newcommand{\B}{\mathcal{B}}
\newcommand{\F}{\mathcal{F}}
\newcommand{\G}{\mathcal{G}}
\newcommand{\E}{\mathcal{E}}
\renewcommand{\H}{\mathcal{H}}
\newcommand{\I}{\mathcal{I}}
\newcommand{\J}{\mathcal{J}}

\newcommand{\R}{\mathcal{R}}
\renewcommand{\S}{\mathcal{S}}
\newcommand{\V}{\mathcal{V}}

\newcommand{\X}{\mathcal{X}}
\DeclareMathOperator*{\argmin}{arg\,min}
\newcommand{\RR}{\mathbb{R}}

\newcommand{\sgn}[1]{\mathrm{sgn}\left(#1\right)}
\newcommand{\fin}{f^{\mathrm{in}}}
\newcommand{\fout}{f^{\mathrm{out}}}
\newcommand{\rhomax}{B}
\newcommand{\fmax}{C}
\newcommand{\be}{\begin{equation}}
\newcommand{\ee}{\end{equation}}

\newcommand{\ds}{\displaystyle}
\newcommand{\ba}{\begin{array}}\newcommand{\ea}{\end{array}}
\newcommand{\de}{\mathrm{d}}
\newcommand{\margindef}[1]{}

\newcommand{\zerobf}{\boldsymbol{0}	}

\newcommand{\defineas}{:=}

\newtheorem{theorem}{Theorem}
\newtheorem{proposition}{Proposition}
\newtheorem{definition}{Definition}

\newtheorem{lemma}{Lemma}
\newtheorem{remark}{Remark}
\newtheorem{example}{Example}

\title{
Stability Analysis and Control Synthesis\\ for Dynamical Transportation Networks\thanks{The first three authors were partially supported by the Swedish Research Council through the Junior Research Grant Information Dynamics in Large Scale Networks and the Linnaeus Excellence Center, LCCC.}}

\author{Enrico Lovisari\thanks{E. Lovisari (corresponding author) is with the GIPSA-lab, UMR-CNRS 5216, and the NeCS team at INRIA Grenoble Rh\^{o}ne-Alpes, Grenoble, France, {\tt \small enrico.lovisari@gipsa-lab.fr}}, Giacomo Como\thanks{G. Como and A. Rantzer are with the Department of Automatic Control, Lund University, SE-221 00 Lund, Sweden {\tt\small {giacomo.como, anders.rantzer}@control.lth.se}.}, Anders Rantzer\footnotemark[3], and Ketan Savla\thanks{K. Savla is with the Sonny Astani Department of Civil and Environmental Engineering, University of Southern California, Los Angeles, CA 90089-2531 {\tt\small ksavla@usc.edu}}
}
\begin{document}
\maketitle
\begin{abstract}

We study dynamical transportation networks in a framework that includes extensions of the classical Cell Transmission Model to arbitrary network topologies. The dynamics are modeled as systems of ordinary differential equations describing the traffic flow among a finite number of cells interpreted as links of a directed network. Flows between contiguous cells, in particular at junctions, are determined by merging and splitting rules within constraints imposed by the cells' demand and supply functions as well as by the drivers' turning preferences, while inflows at on-ramps are modeled as exogenous and possibly time-varying. 

First, we analyze stability properties of dynamical transportation networks. We associate to the dynamics a state-dependent dual graph whose connectivity depends on the signs of the derivatives of the inter-cell flows with respect to the densities. Sufficient conditions for the stability of equilibria and periodic solutions are then provided in terms of the connectivity of such dual graph. 

Then, we consider synthesis of optimal control policies that use a combination of turning preferences, scaling of the demand functions through speed limits, and thresholding of supply functions, in order to optimize convex objectives. We first show that, in the general case, the optimal control synthesis problem can be cast as a convex optimization problem, and that the equilibrium of the controlled network is in free-flow. If the control policies are restricted to speed limits and thresholding of supply functions, then the resulting synthesis problem is still convex for networks where every node is either a merge or a diverge junction, and where the dynamics is monotone. These results apply both to the optimal selection of equilibria and periodic solutions, as well as to finite-horizon network trajectory optimization. 

Finally, we illustrate our findings through simulations on a road network inspired by the freeway system in southern Los Angeles. 
\end{abstract}
%

\section{Introduction}
\label{section:Introduction}

Transportation systems are vital for the well-functioning of the society and the economy. Increasing travel demand combined with limited growth in physical transportation infrastructure necessitates efficient management of transportation systems, by leveraging rapid advancements in sensing and information technologies. The true potential of these technologies can be best utilized within a dynamical framework. 

This paper deals with the stability analysis and control synthesis for road transportation networks. The dynamical models for such systems are primarily classified either as \emph{microscopic} or \emph{macroscopic}. 
Microscopic models describe the behavior of every single vehicle and, because of their complexity, their practical usage is typically limited to small scale systems, e.g., a collection of few intersections. On the other hand, macroscopic models, which are the subject of this paper, describe traffic flow at an aggregate level. The most celebrated macroscopic model is the Lighthill-Whitham and Richards (LWR) model \cite{LighthillTrafficPTRS55}, which describes traffic flow dynamics on a line by a partial differential equation. These models have also been extended to networks, e.g., see \cite{Garavello.Piccoli:06}, by careful consideration of the boundary conditions at the nodes. Space discretization schemes for numerical implementation of PDE models for traffic have also been developed. 
The Cell Transmission Model (CTM) \cite{Daganzo:94,Daganzo:95} is arguably the best known among these discretization schemes.

Inspired by spatially discrete models, including CTM, we model the layout of a transportation system by a directed graph. The links of this graph correspond to \emph{cells} whose direction is aligned to the one of the traffic flow, while its nodes correspond to interfaces between two cells or junctions, e.g., between an on-ramp and main line of a freeway. Every cell is endowed with a demand and a supply function, representing the maximum outflow and the maximum inflow on the cell given its density, respectively. 
We model traffic dynamics by a system of ordinary differential equations (ODEs) representing mass conservation on the cells. Inflows at on-ramps are modeled as exogenous and possibly time-varying.  On the other hand, flows between contiguous cells, in particular at junctions, are determined by merging and splitting rules within constraints imposed by the cells' demand and supply functions as well as by the drivers' turning preferences. In the free-flow regime, i.e., when the supply on every outgoing cell at a junction is less than the cumulative demand from the incoming cells, these rules are specified via turning preferences. For the non-free-flow or congested regime, there are several models in literature. Our model includes and extends several of these models, including FIFO~\cite{Daganzo:95}, non-FIFO~\cite{Karafyllis.Papageorgiou:TCNS14} and priority rules~\cite{Daganzo:95}. 

We prove a general result concerned with local stability of free-flow equilibria under any such rules. Furthermore, we show that the flow dynamics naturally induce a state-dependent dual graph whose connectivity depends on the signs of the derivatives of the inter-cell flows with respect to the densities. If the dynamics satisfy a certain monotonicity property, then, for constant and periodic inflows, we provide sufficient conditions for stability of equilibria and periodic solutions, respectively, in terms of connectivity of the dual graph. Our analysis relies on an $\ell_1$ contraction principle for monotone dynamical systems with mass conservation, which is similar to the corresponding property of the entropy solutions of scalar conservation laws, including LWR models, e.g., cf.~Kru\v{z}kov's Theorem \cite[Proposition 2.3.6]{Serre:99}. 

Then, we consider synthesis of optimal control policies that use a combination of turning preferences, scaling of the demand functions through speed limits, and thresholding of supply functions, in order to optimize convex objectives. We first show that, in the general case, the optimal control synthesis problem can be cast as a convex optimization problem, and that the equilibrium of the controlled network is in free-flow. If the control policies are restricted to speed limits and thresholding of supply functions, then the resulting synthesis problem is still convex for networks where every node is either a merge or a diverge junction, and where the dynamics is monotone. These results apply both to the optimal selection of equilibria and periodic solutions, as well as to finite-horizon network trajectory optimization. 

Part of this paper builds upon our previous work~\cite{ComoPartITAC13,ComoPartIITAC13,Como.Lovisari.ea:TCNS13} on dynamical flow networks, adapting the stability analysis to the standard setting for transportation networks. In particular, the models in \cite{ComoPartITAC13,ComoPartIITAC13,Como.Lovisari.ea:TCNS13} only include demand constraints, but they do not allow for either supply constraints or hard constraints induced by the drivers' turning preferences. 
On the other hand, the optimal control synthesis is a novel feature of the present work that was absent in \cite{ComoPartITAC13,ComoPartIITAC13,Como.Lovisari.ea:TCNS13}. 
This paper extends and unifies existing results on stability analysis for line topology in \cite{GomesTRC08, Pisarski.CAnudas-de-Wit:ACC12}, and for the network case in \cite{Karafyllis.Papageorgiou:TCNS14,CooganACC14}. 

While speed limits, e.g., see \cite{HegyiTRC05}, and metering, e.g., see \cite{GomesTRC08}, have been used as control mechanisms before, we also consider controlling turning preferences for congestion regimes. Although changes in turning preferences can occur naturally because drivers change their route choices when exposed to congestion, one could complement such changes favorably further by providing real-time traffic information to the drivers. 
Computational complexity of optimal equilibrium selection and control problems for transportation have been considered before, primarily in the context of receding horizon control, due to its impact on realistic implementation of such control schemes. Existing strategies are based on Mixed Integer Linear Program formulations \cite{LinTITS11, FrejoTRC:14} or on relaxation of the problem to obtain linear formulations \cite{MuralidharanACC12}. A recent approach relies on avoiding the discretization of the underlying LWR model via CTM and yields a reduction of the number of control variables \cite{LiTCNS:14}, but requires affine initial and boundary conditions. To the best of our knowledge, one of the novelties of this paper is to identify sufficient conditions for convexity of optimal control problems in the general network setting.

The major contributions of this paper are as follows. First, we propose a dynamical model for transportation networks that extends several well-known models to networks with arbitrary topologies. We introduce a state-dependent dual graph whose connectivity gives sufficient condition for stability of equilibria and periodic solutions when the dynamics is monotone. We postulate the problem of optimal control synthesis for transportation networks, and identify conditions under which it is a convex optimization problem. 
Finally, we illustrate our findings through simulations on a road network inspired by the freeway system in southern Los Angeles. 

The paper is organized as follows: in the rest of this section we provide some basic notations. In Section~\ref{section:model} we describe the dynamical transportation network model and illustrate how several well-known models fit into this framework. In Section~\ref{sec:stability}, we introduce the notion of dual graph and show its connection to stability of equilibria and periodic solutions for monotone dynamics. Section~\ref{sec:control} is devoted to the optimal control synthesis problem. In Section~\ref{section:numericalExample}, we present results from simulation studies, and Section~\ref{section:conclusions} draws the conclusions. 
For completeness, we briefly summarize key concepts from nonlinear dynamical systems relevant for this paper in Appendix~\ref{dynamicalSystems}, while in Appendix~\ref{technicalresults} we state a result from our previous work that is used a few times in the paper. The proof of Lemma~\ref{lemma:l1strictlydecreasing} is also given in Appendix~\ref{technicalresults}.

\subsection{Notation}

The symbols $\RR$ and $\RR_+ \defineas \{x\in\RR:\,x\ge0\}$ denote the set of real and nonnegative real numbers, respectively. For finite sets $\A$ and $\B$, $|\A|$ denotes the cardinality of $\A$, $\RR^{\A}$ (respectively, $\RR_+^{\A}$) the space of real-valued (nonnegative-real-valued) vectors whose components are indexed by elements of $\A$, and $\RR^{\A\times\B}$ the space of matrices whose real entries are indexed by pairs in $\A\times\B$. 

The transpose of a matrix $M \in \RR^{\A \times\B}$ is denoted by $M' \in\RR^{\B\times\A}$, while $\zerobf$ stands for the all-zero vector of suitable dimension.
 The natural partial ordering of $\RR^{\A}$ will be denoted by $x \leq y$ for two vectors $x,y\in\RR^{\A}$ such that $x_a\le y_a$ for all $a\in\A$.

\begin{figure}
	\centering
		\includegraphics[width=8cm]{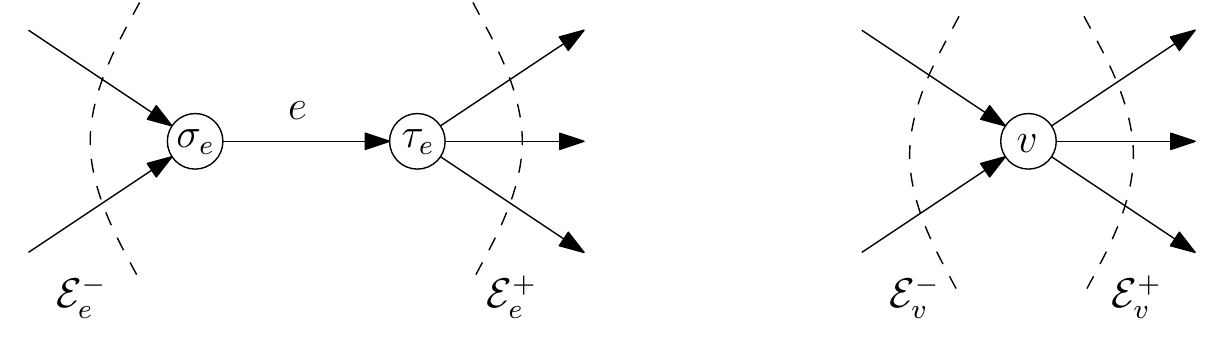}
	\caption{Illustration of some notations.}
	\label{fig:commonNotationRoundabout}
\end{figure}

A directed multi-graph is a couple $\G = (\V, \E)$, where $\V$ and $\E$ stand for the node set and the link set, respectively, and are both finite. They are endowed with two vectors: $\sigma,\tau\in\V^{ \E }$. For every $e\in\E$, $\sigma_e$ and $\tau_e$ stand for the tail and head nodes respectively of link $e$. We shall always assume that there are no self-loops, i.e., $\tau_e\ne\sigma_e$ for all $e \in \E$. On the other hand, we allow for parallel links. For a node $v \in \V$, let $\E_v^+ \defineas \{e:\sigma_e = v\}$ and $\E_v^- \defineas \{e:\tau_e = v\}$. For a link $e \in \E$, let $\E_e^+ \defineas \E_{\tau_e}^+$ be the set of links downstream to $e$ and $\E_e^- \defineas \E_{\sigma_e}^-$ be the set of links upstream to $e$. See Figure~\ref{fig:commonNotationRoundabout} for an illustration of some of these notations. Finally, for brevity in notation, unless explicitly specified otherwise, the range of indices under summation is understood to be the set $\E$.


\section{Dynamical transportation networks}
\label{section:model}

We model dynamical transportation networks as systems of ODEs of the form 
	\be
	\label{eq:systfirst}
	\dot{\rho}_i= \fin_i(\rho,t) - \fout_i(\rho, t)\,,\qquad i\in\E\,,
	\ee
often written in compact form $\dot{\rho} = g(\rho, t)$. Here, $\rho_i\geq0$ stands for the density on a cell $i$ with $\E$ denoting the set of all cells, $\rho \in\RR^\E$ stands for the vector of all densities on the different cells, and $\fin_i(\rho,t)$ and $\fout_i(\rho,t)$ denote the inflow to and, respectively, the outflow from cell $i$. 

\begin{remark}
\label{rem:densityVSvolume}
In order to be consistent with existing literature, the presentation in Sections \ref{section:model} through \ref{sec:control} is in terms of densities on cells. Consequently, \eqref{eq:systfirst} is a mass conservation law if and only if, as we implicitly assume, all cells have the same length. Note that, this is without loss of generality, since the case of heterogeneous cells can be treated similarly by interpreting the variable $\rho_i(t)$ as the volume of vehicles, rather than density, on cell $i$ at time $t$. We adopt this approach in Section~\ref{section:numericalExample}.
\end{remark}

Cells are meant to represent portions of roads as well as on- and off-ramps. Following Daganzo's seminal work \cite{Daganzo:94,Daganzo:95}, the physical characteristics of each cell $i$ are captured by a possibly time-dependent demand function $d_i(\rho_i,t)$ and a supply function $s_i(\rho_i,t)$, representing upper bounds on the outflow from and, respectively, the inflow in cell $i$ at time $t$, when the current density on it is $\rho_i$, i.e., 
	\be\label{eq:supplydemandconstraints}
		\fin_i(\rho,t)\le s_i(\rho_i,t)\,,\qquad \fout_i(\rho,t)\le d_i(\rho_i,t)\,,\qquad i\in\E\,,\  \rho_i\ge0\,,\ t\ge0\,.
	\ee
Throughout, we assume that, on every cell $i\in\E$, the demand function $d_i(\rho_i,t)$ is strictly increasing in $\rho_i$, and the supply function $s_i(\rho_i,t)$ is non-increasing in $\rho_i$, for all $t\geq0$.
\begin{figure}
	\centering
		\includegraphics[width=5cm]{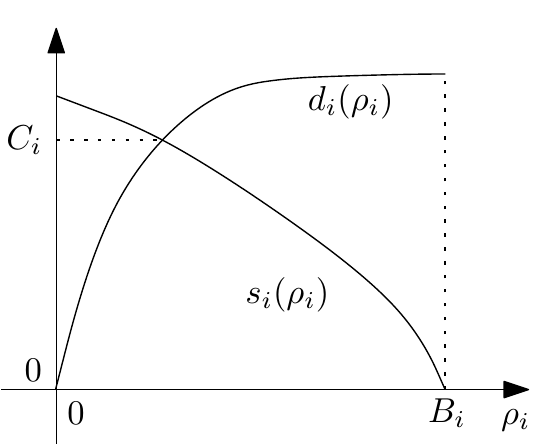}
	\caption{Illustration of demand and supply functions, and of flow capacity, on a cell $i\in\E$.}
	\label{fig:SupplyAndDemandExample}
\end{figure}

The time-invariant \emph{jam density} on cell $i$, namely, the maximum density allowed on $i$, is defined as\footnote{Since the jam density does not depend on time, $t$ here is arbitrary.}
	$$
		\rhomax_i := \sup\{\rho\geq0:\,s_i(\rho,t) > 0\}\,.
	$$ 
Observe that \eqref{eq:supplydemandconstraints} implies that the set $\S := \prod_{i\in\E}[0,\rhomax_i]$ is invariant under the dynamics \eqref{eq:systfirst}, and, in accordance with the physics of the system, it is assumed throughout the paper that the state of the network belongs to $\S$ at any time. 

The function $q_i(\rho_i,t)=\min\{d_i(\rho_i,t),s_i(\rho_i,t)\}$ is often interpreted as the \emph{fundamental diagram} and $C_i(t):=\max_{\rho_i\ge0}q_i(\rho_i,t)$ as the \emph{flow capacity} of cell $i$. See Figure~\ref{fig:SupplyAndDemandExample} for an illustration. A particularly relevant role in the applications has been played by the special case of linear demand functions 
	\begin{equation}
		\label{eq:linearDemand}
		d_i(\rho_i,t) = v_i(t)\rho_i\,,
	\end{equation}
where $v_i(t)$ is  the \emph{free-flow speed} and saturated affine supply functions 
	\be
		\label{eq:affinesupply}
		s_i(\rho_i,t) = \min\{S_i(t),w_i(t)(B_i-\rho_i)\}\,,
	\ee
where $w_i(t)$ is the wave-speed and $S_i(t)$ a supply saturation level. 

In some cases, the free-flow speed $v_i(t)$ as well as the wave speed $w_i(t)$ and/or the saturation level of the supply function $S_i(t)$ can be considered as control parameters that can be actuated, e.g., through variable speed limits and, respectively, supply metering. We refer the reader to Section \ref{sec:control} for a more in depth  discussion on this, along with design principles.  In other cases, time variability of such parameters may be thought of as due to uncontrolled perturbations of the system due to, e.g., increased inflows in some parts of the network, changing weather conditions, accidents, and so on. 

Cells are conveniently identified with the links of a directed graph $\G=(\V,\E)$ whose nodes $v\in\V$  represent junctions between consecutive cells. Conventionally, we include in the set $\V$ an extra node $w$ meant to represent the external world, which is both the source of the flow entering the network and the sink of flow exiting it. 
In particular, cells $i\in\E$ such that $\sigma_i = w$ represent on-ramps, while cells $j\in\E$ such that $\tau_j = w$ represent off-ramps. We denote the sets of on- and off-ramps by $\R \subseteq\E$ and $\R^o\subseteq \E$, respectively. To avoid trivial cases, we will always assume  that 
for every cell $i$ there exists at least one directed path from $i$ to an off-ramp $j\in\R^o$. We will also use the convention that on all on-ramps $i\in\R$ the supply and jam density are infinite, i.e., $s_i(\,\cdot\,,\,\cdot\,)\equiv+\infty$ and $B_i=+\infty$. In contrast, we will assume that the jam density on every other cell is finite, i.e., $B_i<+\infty$ for all $i\in\E\setminus\R$.

The network topology described by the directed graph $\G$ induces natural constraints on the dynamics \eqref{eq:systfirst}: flow is possible only between consecutive cells, i.e., from a cell $i$ to a cell $j$ such that $\tau_i=\sigma_j$. 
Specifically, we model the flow $f_{ij}(\rho,t)$ from a cell $i\in\E\setminus\R^o$ to a cell $j\in\E \setminus\R$ as a continuous function of $(\rho,t)$, Lipschitz-continuous in $\rho$ uniformly continuous with respect to $t$, and let 
the inflow in and outflow from a cell $i$, respectively, satisfy the following equations
	\begin{align}
		\nonumber		
		\fin_i(\rho,t) &= 
			\begin{cases}
				\lambda_i(t), &	i\in\R	\\[5pt]
				\sum_{j\in\E_i^-}f_{ji}(\rho,t), & i\in\E\setminus\R
			\end{cases}
			\\
			\label{eq:finfout}
			\\
		\nonumber\fout_i(\rho,t) &= 
			\begin{cases}
				d_i(\rho_i,t), &	i\in\R	^o\\[5pt]
				\sum_{j\in\E_i^+}f_{ij}(\rho,t), & i\in\E\setminus\R^o
			\end{cases}		
	\end{align}
Equation \eqref{eq:finfout} states that 
	\begin{itemize}
		\item the inflow $\fin_i(\rho,t) = \lambda_i(t)$ into an on-ramp $i\in\R$ is independent of the state of the network. Here, $\lambda_i(t)\ge0$ has to be interpreted as an input to the network, modeling the possibly time-varying rate at which vehicles enter the on-ramp $i$ from the external world. Throughout, we shall denote the vector of inflows at the different on-ramps by $\lambda(t)\in\RR^{\R}$, $\lambda(t)\geq0$, and assume that it is a piecewise continuous function of time for all $t\ge0$. 
		\item the outflow $\fin_i(\rho,t) = d_i(\rho,t)$ from an off-ramp $i\in\R^o$ equals the demand on cell $i$;
		\item besides the two cases above, inflow and outflow of a cell $i$ coincide with the sum of all incoming flows from cells  $j\in\E^-_i$ immediately upstream and, respectively, outgoing flows to the cells $j\in\E^+_i$ immediately downstream, of cell $i$.
	\end{itemize}

Standard results on ODEs guarantee, under the given continuity assumptions on the flow functions $f_{ij}$ and on the input vector $\lambda(t)$, the existence and uniqueness of a solution to \eqref{eq:systfirst} for every initial condition $\rho(0)\in \S$. 

The actual form of the flow functions $f_{ij}(\rho,t)$ depends on a possibly time-dependent \emph{turning preference matrix} $R(t)\in\RR^{\E\times\E}$ satisfying, for all $t\ge0$, and $i,j\in\E$, \footnote{Recall that the range of summation is $\E$, unless specified otherwise.}
	$$
		R_{ij}(t)\ge0\,,
		\qquad
		\sum_kR_{ik}(t)=1\,,
		\qquad 
		\tau_i\ne\sigma_j\ \Rightarrow R_{ij}(t)=0	
		\,.
	$$

In particular, the flow functions have to satisfy the following natural constraints 
	\be	
		\label{eq:constraintFetoj}
		 0 \leq  f_{ij}(\rho,t)\le R_{ij}(t)d_i(\rho_i,t)\,,\qquad \forall i,j\in\E
	\ee
	\be 	
		\label{eq:constraintFin}
		\sum_kf_{ki}(\rho,t)\le s_i(\rho_i,t)\,, \forall i\in\E
	\ee
and, for all $v\in\V$,
	\be
		\label{eq:constraintFetojEquality} 
		\sum_k R_{kj}(t)d_k(\rho_k,t)\le s_j(\rho_j,t),\,\quad \forall j\in\E_v^+\\
	\quad\Longrightarrow\quad f_{ij}(\rho,t)= R_{ij}(t)d_i(\rho_i,t), \quad\forall i\in\E_v^-\,.
	\ee
The constraints \eqref{eq:constraintFetoj} and \eqref{eq:constraintFin} ensure that the flow from a cell $i$ to another cell $j$ never exceeds the part of the demand on $i$ whose preference is to turn to $j$ and, respectively,  that the total inflow in cell $i$ never exceeds its supply. In particular, flow among non-consecutive cells is zero. On the other hand, \eqref{eq:constraintFetojEquality} states that, when there is enough supply on all cells downstream of node $v$, the flow from a cell $i\in\E_v^-$ to a cell $j\in\E_v^+$ coincides with the part of the demand on $i$ whose preference is to turn to $j$. Hence, $R_{ij}(t) \geq 0$ represents the fraction of the demand on cell $i$ that flows to cell $j$ when there is enough supply on all cells in $\E^+_{\tau_i}$. In particular, we say that cell $i$ is in \emph{free-flow} when $f_{ij}(\rho,t)= R_{ij}(t)d_i(\rho_i,t)$ for all $j\in\E_i^+$.

Observe that the constraints \eqref{eq:constraintFetoj}--\eqref{eq:constraintFetojEquality} do not uniquely characterize the value of the flow functions $f_{ij}(\rho,t)$ when 
	\be
		\label{eq:supplyviolated}
		\sum_k R_{kj}(t)d_k(\rho_k,t)>s_j(\rho_j,t)\,.
	\ee 
In this case, \eqref{eq:constraintFin} only ensures that the total inflow in $j$ does not exceed the supply of cell $j$, while specific allocation rules are needed to determine how much of such supply is allocated to the flows from the different cells $i\in\E^-_j$. 
We now present a few different examples of flow functions. 

\begin{example}\label{example:line}
Consider a network with line topology, consisting of $N$ consecutive cells. The first and the last cells are an on-ramp and an off-ramp, respectively. The dynamics is given by
	\begin{align*}
		\dot{\rho}_1(t) & = \lambda(t) - f_{1,2}(\rho_1(t), \rho_2(t), t)	\\
		\dot{\rho}_i(t) & = f_{i-1,i}(\rho_{i-1}(t), \rho_i(t),t) - f_{i,i+1}(\rho_i(t), \rho_{i+1}(t),t), \quad \forall i = 2,\dots, N-1\\
		\dot{\rho}_N(t) & = f_{N-1,N}(\rho_{N-1}(t), \rho_N(t),t) - d_N(\rho_N(t))\\
	\end{align*}
where $\lambda(t)$ is the inflow at the on-ramp. Under the Cell Transmission Model (CTM)~\cite{Daganzo:94}, the flow from one cell to the next is given by
	$$
		f_{i,i+1}(\rho_i, \rho_{i+1}, t) = \min\{d_i(\rho_i, t), s_{i+1}(\rho_{i+1}, t)\}\,.
	$$
Since in this case $R_{i,i+1} = 1$ for $i= 1,\dots, N-1$ and $R_{ij} = 0$ otherwise, this policy satisfies \eqref{eq:constraintFetoj}, \eqref{eq:constraintFin} and \eqref{eq:constraintFetojEquality}.
\end{example}

\begin{example}\label{example:policiesFIFO}
One possible extension of the CTM to the network setting is the First In First Out (FIFO) policy \cite{Daganzo:95}\cite{CooganACC14}. Given a node $v \in\V$, $i\in\E_v^-$ and $j\in\E_v^+$, under the FIFO policy,
	$$
		f_{i j}(\rho) = \kappa_v^{\mathrm{F}}(\rho,t) R_{ij}(t)d_i(\rho_i,t)
	$$
where $\kappa_v^{\mathrm{F}}(\rho,t) \in [0,1]$ is the maximum value such that $\kappa_v^{\mathrm{F}}(\rho,t)\sum_{i\in\E_v^-}R_{ik}(t)d_i(\rho_i,t)\leq s_k(\rho_k,t)$ for all $k\in\E_v^+$. In words, the FIFO policy corresponds to a situation where drivers are rigid about their route choice, and all the cells consist of a single lane. Therefore, increase in congestion on any outgoing cell at a node potentially decreases inflow to other outgoing cells. The FIFO policy gives maximum flow between incoming and outgoing cells subject to such constraints. It is straightforward to see that such policy satisfies \eqref{eq:constraintFetoj}, \eqref{eq:constraintFin} and \eqref{eq:constraintFetojEquality}. 
\end{example}

\begin{example}
\label{example:policiesNonFIFO}
A second possible extension of the CTM to the network setting is the following non-FIFO proportional allocation rule, which appears in \cite{Karafyllis.Papageorgiou:TCNS14}:
	\begin{equation}
		\label{eq:nonFIFO}
		f_{ij}(\rho,t)
		\!=\!
		\kappa_j^{\mathrm{NF}}(\rho,t)		
		R_{ij}(t)d_i(\rho_i,t)
	\end{equation}		
where
	$$
	\kappa_j^{\mathrm{NF}}(\rho,t)	=
		\min\!\left\{1, \frac{s_j(\rho_j,t)}{\sum_{k}R_{kj}(t)d_k(\rho_k,t)}\right\}\,.
	$$
Notice that $\kappa_v^{\mathrm{F}}(\rho) = \min_{j\in\E_v^+}\kappa_j^{\mathrm{NF}}(\rho)$. Unlike the FIFO policy in Example~\ref{example:policiesFIFO}, under the non-FIFO policy, inflows to the cells outgoing from a node are independent, possibly because of a combination of having multiple lanes and adaptive route choice of drivers. This policy also satisfies \eqref{eq:constraintFetoj}, \eqref{eq:constraintFin} and \eqref{eq:constraintFetojEquality}.
\end{example}

\begin{example}\label{example:mixture}
The FIFO and non-FIFO policies from Examples~\ref{example:policiesFIFO} and \ref{example:policiesNonFIFO}, respectively, possibly represent two extremes of traffic splitting at a congested intersection, and that practical settings correspond to somewhere in between. Accordingly, we propose the following \emph{mixture} model: 
	\begin{equation}
		\label{eq:mixture}
		f_{ij}(\rho,t)
		\!=\!
		\kappa_j^{\mathrm{NF}}(\rho)		
		R_{ij}(t)d_i(\rho_i,t)
	\end{equation}		
for $i\in\E_v^-$ and $j\in\E_v^+$, and 
	\begin{equation}
	\label{eq:mixture-coefficient}
		\kappa_j^{\mathrm{M}}(\rho)	 = \theta\kappa_v^{\mathrm{F}}(\rho)	+ (1-\theta)\kappa_j^{\mathrm{NF}}(\rho)
	\end{equation}
where $\theta \in [0,1]$ is the mixture parameter. For any $\theta$, this policy satisfies \eqref{eq:constraintFetoj}, \eqref{eq:constraintFin} and \eqref{eq:constraintFetojEquality}.
\end{example}

\begin{example}
\label{example:policiesDaganzo}
Consider the setting where every node is either a merge node, i.e., having a single outgoing cell, or a diverge node, i.e., having a single incoming cell. Further, let every merge node have at most two incoming cells, e.g., when the node corresponds to a junction of the mainline of a freeway and an on-ramp. Consider one such merge node with $j$ as the unique outgoing cell, and $i$ and $k$ as the two incoming cells. Let $f_{ij}(\rho,t)$ be given by the following priority rule proposed in \cite{Daganzo:95}:
	\begin{align*}
	\label{eq:policyDaganzo}
			&f_{i j}(\rho,t) = d_i(\rho_i,t), \textrm{ if }d_i(\rho_i,t)+d_k(\rho_k,t) \leq s_j(\rho_j,t)
	\end{align*}
and
	\begin{align*}
		&f_{ij}(\rho,t)=		
		\textrm{mid}\{d_i(\rho_i,t), s_j(\rho_j,t) - d_k(\rho_k,t), p_is_j(\rho_j,t)\},\\
		&	\qquad\qquad\qquad\quad\textrm{ if }d_i(\rho_i,t)+d_k(\rho_k,t) > s_j(\rho_k,t)
	\end{align*}
and symmetrically for $f_{k j}(\rho,t)$. Here $\textrm{mid}\{a,b,c\}$ denotes the middle value among $a$, $b$, and $c$, and $p_i$, $p_k$ are nonnegative and such that $p_i+p_k = 1$. The higher $p_i/p_k$, the more priority is given to $i$ with respect to $k$ under congestion, i.e., when $d_i(\rho_i,t)+d_k(\rho_k,t) > s_j(\rho_k,t)$. Finally, for diverge nodes, consider one of the rules described in the previous examples. Such a policy satisfies \eqref{eq:constraintFetoj}, \eqref{eq:constraintFin} and \eqref{eq:constraintFetojEquality}. 
\end{example}

\section{Stability analysis}
\label{sec:stability}

This section is devoted to the stability analysis of dynamical transportation networks. We refer to Appendix~\ref{dynamicalSystems} for a primer on 
key concepts from nonlinear dynamical systems that are relevant for this section. 
We start with the following simple result, which, e.g., appeared in  \cite{CooganACC14}. It shows that when \emph{any} transportation network admits an equilibrium that is in free-flow, then such an equilibrium is locally asymptotically stable, i.e., it attracts all sufficiently close points. We give a proof for completeness and because it will also give us some insights on the properties of stable equilibria.

\begin{proposition}
\label{prop:freeflowGeneral}
Let \eqref{eq:systfirst} be a dynamical transportation network with time-invariant demand $d_i(\rho_i)$ and supply function $s_i(\rho_i)$ on every cell $i\in\E$, and constant turning preference matrix $R$ and inflow vector $\lambda\in\RR^{\R}_+$. Extend $\lambda$ to a non-negative vector in $\RR^\E$
by putting $\lambda_i=0$ for all $i\in\E\setminus\R$, and let $C\in\RR^\E$ be the vector of cells' flow capacity. Then, if  
	\begin{equation}
		\label{eq:conditionFreeflow}
		f^* := (I - R^T)^{-1}\lambda < C\,,	
	\end{equation}
then $\rho^*\in\S$ defined by 
	$$
		\rho^*_i = d_i^{-1}(f^*_i) \,, \qquad i\in\E
	$$
is a free-flow locally asymptotically stable equilibrium. 
\end{proposition}
\begin{proof}
First we prove that $\rho^*$ is an equilibrium for the system. Indeed, $f^*_j < \fmax_j$ implies, by definition of flow capacities and by the properties of demand and supply functions, that $d_j(\rho_j^*) < s_j(\rho_j^*)$. Then
	\begin{align*}
		\sum_{e\in\E_j^-}R_{ej}d_e(\rho_e^*) 
			& 	= \sum_{e\in\E_j^-}R_{ej}f^*_e 
				= f^*_j = d_j(\rho_j^*) < s_j(\rho_j^*)
	\end{align*}
where the second equality follows by \eqref{eq:conditionFreeflow}. By \eqref{eq:constraintFetojEquality}, we obtain $f_{ej}(\rho^*) = R_{ej}d_e(\rho_e^*) = R_{ej}f^*_e$ for all $(e,j)$. Together with $R^Tf^* + \lambda = 0$, this establishes that $\rho^*$ is a free-flow equilibrium.

We now prove that $\rho^*$ is locally asymptotically stable. In fact, let $J = \nabla g(\rho)|_{\rho = \rho^*}$ be the Jacobian of the system computed at $\rho^*$. Since $f_{ej}(\rho^*) = R_{ej}d_e(\rho_e^*)$ for all $(e,j)$, then for all $(e,j)$ with $e\neq j$ it holds $J_{ej} = R_{ej}\frac{\partial d_e(\rho_e)}{\partial \rho_e}|_{\rho_e=\rho_e^*}$ (this being zero if $\tau_e\neq\sigma_j$), while $J_{ee} = -\frac{\partial d_e(\rho_e)}{\partial \rho_e}|_{\rho_e=\rho_e^*}$ for all $e$. This implies that $J$ is a Metzler matrix, whose columns all have non positive sum, and in particular whose columns corresponding to off-ramps have strictly negative sum. Moreover, by assumption in the original graph for every cell there is at least one directed path to at least one off-ramp. Under these assumptions, it is well known (see, e.g., \cite{Lovisari.Como.ea:CDC14}[Lemma 7]) that $J$ is a stable matrix. Therefore, $\rho^*$ is a locally stable equilibrium.
\end{proof}

For FIFO policies such as those presented in Example~\ref{example:policiesFIFO} it can be shown \cite{CooganACC14} that the basin of attraction of $\rho^*$ contains any $\rho\in\S$ such that $0\leq \rho \leq \rho^*$. However, in general, the free-flow equilibrium of a transportation network with FIFO policy is not globally asymptotically stable. This is shown in following example, and also illustrated in our simulation studies in Section~\ref{section:numericalExample}.

\begin{figure}
\begin{center}
\includegraphics[scale=0.75]{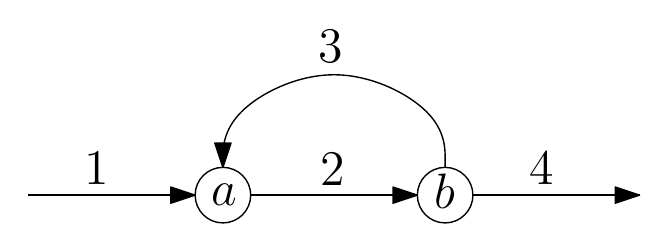}
\end{center}
\caption{The network considered in Example~\ref{example:counterexampleGASFIFO}}
\label{figure:CounterexampleGASFIFO}
\end{figure}

\begin{example}
\label{example:counterexampleGASFIFO}
Consider a the network with four cells, one on-ramp and one off-ramp shown in Figure~\ref{figure:CounterexampleGASFIFO}. Let the dynamics be driven by a FIFO policy as per Example~\ref{example:policiesFIFO}, so that
	{\small
	\begin{align*}
		\dot{\rho}_1 & = g_1(\rho) = \lambda - \kappa_a^F(\rho)d_1(\rho_1)	&
		\dot{\rho}_2 & = g_2(\rho) = \kappa_a^F(\rho)(d_1(\rho_1)+d_3(\rho_3)) - \kappa_b^F(\rho)d_2(\rho_2)	\\
		\dot{\rho}_3 & = g_3(\rho) = \kappa_b^F(\rho)R_{23}d_2(\rho_2) - \kappa_a^F(\rho)d_3(\rho_3)&
		\dot{\rho}_4 & = g_4(\rho) = \kappa_b^F(\rho)R_{24}d_2(\rho_2) - d_4(\rho_4)
	\end{align*}
	}
where 
	$$
		\kappa_a^F(\rho) = \min\left\{1, \frac{s_2(\rho_2)}{d_1(\rho_1)+d_3(\rho_3)}\right\}\,,\quad
		\kappa_b^F(\rho) = \min\left\{1, \frac{s_3(\rho_3)}{R_{23}d_2(\rho_2)}, \frac{s_4(\rho_4)}{R_{24}d_2(\rho_2)}\right\}\,.
	$$
	
Let $d_e(\rho_e) = \rho_e$ for all $e$ and $R_{23} = R_{24} = 0.5$. The candidate free-flow equilibrium can be easily found to be $\rho_1^* = \rho_3^* = \rho_4^* = \lambda$ and $\rho_2^* = 2\lambda$. Let the supply functions be so that $\rho^*$ is indeed an equilibrium in the free flow. By Proposition~\ref{prop:freeflowGeneral}, $\rho^*$ is locally asymptotically stable. In order to see that it is not globally asymptotically stable, consider the trajectory $\hat{\rho}_1(t) = \lambda t + \rho^o$, $\rho^o\geq0$ arbitrary, and $\hat{\rho}_2(t) = \rhomax_2$, $\hat{\rho}_3(t) = \rhomax_3$, $\hat{\rho}_4(t) = 0$ for all $t\geq0$. This is a 
feasible trajectory for the network under consideration because $s_2(\hat{\rho}_2(t)) = s_3(\hat{\rho}_3(t)) = 0$, and hence $\kappa_a^F(\hat{\rho}(t)) = \kappa_b^F(\hat{\rho}(t)) = 0$, for all $t\geq0$. Therefore,
	$$
		g_1(\hat{\rho}(t)) = \lambda,\quad	
		g_2(\hat{\rho}(t)) = 0,\quad
		g_3(\hat{\rho}(t)) = 0,\quad
		g_4(\hat{\rho}(t)) = 0,\qquad \forall t\geq0
	$$
proving that $\{\hat{\rho}(t):\,t\geq0\}$ is a trajectory of the system. Clearly, since the densities on cells $2$, $3$ and $4$ do not change, while the density on the on-ramp $1$ grows unbounded, $\hat{\rho}(t)$ does not converge to $\rho^*$, which is thus not globally asymptotically stable.
\end{example}

Example~\ref{example:counterexampleGASFIFO} shows that one cannot prove global asymptotic stability of free-flow equilibria for general transportation networks. Even more so, for the FIFO policies in Example~\ref{example:policiesFIFO}, condition \eqref{eq:conditionFreeflow} (with non strict inequality) is a necessary and sufficient condition for the network to admit an equilibrium. That is, under FIFO policies, if \eqref{eq:conditionFreeflow} does not hold true with non-strict inequality, then no equilibrium can exist and, moreover, the trajectory of the system grows unbounded for any initial condition. 

However, this feature does not extend to \emph{all} transportation networks. Indeed, in the rest of this section, we shall consider a special class of dynamical transportation networks, to be referred to as \emph{monotone}, characterized by an additional differential constraint on the flow functions. Monotone dynamical transportation networks include those with flow functions as in Examples~\ref{example:line} and \ref{example:policiesNonFIFO}, but not the FIFO diverge rule of Example~\ref{example:policiesFIFO} and the mixture model of Example \ref{example:mixture}. 

We will show that, under such additional differential constraint, \eqref{eq:systfirst} is a monotone dynamical system \cite{HirschJMA85}, i.e., one preserving the standard partial ordering in $\S$. Then, we will prove that such monotonicity property combined with the conservation of mass implies a fundamental \emph{incremental stability} property: 	the $l_1$-distance between two solutions of \eqref{eq:systfirst} starting from different initial conditions can never increase. We will explore several consequences of such monotonicity and incremental stability properties including, local stability of free-flow and non free-flow equilibria and periodic solutions and global asymptotic stability of equilibria and periodic solutions, when a certain state-dependent dual graph, which depends on the sign of the derivatives of the flow between contiguous cells, is connected.
This will substantially extend the available results on stability analysis for dynamical transportation networks.
\medskip

We consider flow functions that satisfy the following additional constraint 
				\begin{equation}
					\label{eq:constraintMonotonicity}
						\frac{\partial}{\partial \rho_j} \sum_{k}f_{ki}(\rho,t) \geq 0\,,
						\qquad
						\frac{\partial}{\partial \rho_j} \sum_{k}f_{ik}(\rho,t)\leq 0\,,\qquad\forall j\ne i\in\E\,, 
				\end{equation}
for all $t\ge0$ and almost every $\rho\geq0$\footnote{Lipschitz continuity of the flow functions implies their differentiability almost everywhere by Rademacher's theorem. }.

Equation \eqref{eq:constraintMonotonicity} states that the total inflow in (respectively, outflow from) a cell $i$ does not increase (does not decrease) if the density is increased in any other cell $j\ne i$ and kept constant on cell $i$. We will refer to dynamical systems of the form \eqref{eq:systfirst} satisfying \eqref{eq:finfout}--\eqref{eq:constraintFetojEquality} and \eqref{eq:constraintMonotonicity} as \emph{monotone dynamical transportation networks}. 

\textbf{Examples \ref{example:line}, 	\ref{example:policiesNonFIFO} - contd.}\emph{
The policies presented in Examples~\ref{example:line} and \ref{example:policiesNonFIFO}, which we discuss together as the former is a special case of the latter, are monotone. To see this, first observe that 
	$$
		\sum_kf_{ki}(\rho,t)=\min\left\{s_i(\rho_i,t),\sum_{k}R_{ki}(t)d_k(\rho_k,t)\right\}
	$$			
Then, notice that, for all $j\ne i$, 
	$$
		\frac{\partial}{\partial \rho_j} s_i(\rho_i,t)=0\,,\qquad  \frac{\partial}{\partial \rho_j}\sum_{k}R_{ki}(t)d_k(\rho_k,t)\ge0\,,
	$$
where the second inequality follows from the monotonicity assumption on the demand functions. Hence, the leftmost inequality in \eqref{eq:constraintMonotonicity} is satisfied. Similarly, for every cell $k$, 
	$$
		\frac{\partial}{\partial \rho_j}
		\frac{s_k(\rho_k,t)R_{ik}(t)}{\sum_{l}R_{lk}(t)d_l(\rho_l,t)}\le0
	$$
so that, for all $j\ne i$
	$$
		\frac{\partial}{\partial \rho_j}\sum_kf_{ik}(\rho,t)=
		d_{i}(\rho_i,t)\frac{\partial}{\partial \rho_j}\sum_k
		\min\left\{\frac{s_k(\rho_k,t)R_{ik}(t)}{\sum_{l}R_{lk}(t)d_l(\rho_l,t)},R_{ik}(t)\right\}\le0\,.
	$$
Hence, also the rightmost inequality in \eqref{eq:constraintMonotonicity} is satisfied. Thus the flow functions in the two Examples give rise to a monotone dynamical transportation network. 
}

\textbf{Example \ref{example:policiesFIFO} - contd.}\emph{
In contrast, the policies proposed in Examples~\ref{example:policiesFIFO} and \ref{example:mixture} are not monotone in general. In fact, consider under FIFO policy (Example~\ref{example:policiesFIFO}) a single diverge node $v$ with $\{i\} = \E_v^-$ and $\{j,k\} = \E_v^+$. Then the FIFO policy reads
	\begin{align*}
		f_{ij}(\rho,t) 
			&	= R_{ij}(t)d_i(\rho_i,t)\min\left\{1,\frac{s_j(\rho_j,t)}{R_{ij}(t)d_i(\rho_i,t)},\frac{s_k(\rho_k,t)}{R_{ik}(t)d_i(\rho_i,t)}\right\}		\\
		f_{ik}(\rho,t) 
			&	= R_{ik}(t)d_i(\rho_i,t)\min\left\{1,\frac{s_j(\rho_j,t)}{R_{ij}(t)d_i(\rho_i,t)},\frac{s_k(\rho_k,t)}{R_{ik}(t)d_i(\rho_i,t)}\right\}			
	\end{align*}
so that when $\frac{s_j(\rho_j,t)}{R_{ij}(t)d_i(\rho_i,t)} < \frac{s_k(\rho_k,t)}{R_{ik}(t)d_i(\rho_i,t)} < 1$ we have
	$$
		f_{ik}(\rho,t) = \frac{R_{ik}(t)}{R_{ij}(t)}s_j(\rho_j,t)
	$$
and therefore $\frac{\partial \fin_k(\rho,t)}{\partial \rho_j} < 0$. The system under FIFO policies is thus in general non monotone, as already observed in \cite{CooganACC14}. The same obviously holds for mixed policies.
}

\textbf{Example \ref{example:policiesDaganzo} - contd.}\textit{
Routing under the priority rule proposed in \cite{Daganzo:95} for merge with at most two incoming cells, and the non-FIFO policies in Example~\ref{example:policiesNonFIFO} for diverge is also monotone. We only need to study the merge case, and to this aim consider a node $v$ with $\E_v^- = \{i,k\}$ and $\E_v^+ =\{j\}$. If $d_i(\rho_i,t)+d_k(\rho_k,t) \leq s_j(\rho_j,t)$, then $\fout_i(\rho,t) = f_{ij}(\rho,t) = d_i(\rho_i,t)$ and $\fout_k(\rho,t) = f_{kj}(\rho,t) = d_k(\rho_k,t)$ and the two inequalities in \eqref{eq:constraintMonotonicity} are satisfied. Assume thus $d_i(\rho_i,t)+d_k(\rho_k,t) > s_j(\rho_j,t)$. First of all, notice that
	$$
		\frac{\partial \fout_i(\rho,t)}{\partial \rho_k} 
			=	\frac{\partial f_{ij}(\rho,t)}{\partial \rho_k} 
			\in
			\left\{0, -\frac{\partial d_k(\rho_k,t)}{\partial \rho_k}\right\}
		\,\quad
		\frac{\partial \fout_i(\rho,t)}{\partial \rho_j} 
			\in
			\left\{0, \frac{\partial s_j(\rho_j,t)}{\partial \rho_j},p_i\frac{\partial s_j(\rho_j,t)}{\partial \rho_j}\right\} 
	$$
so the rightmost inequalities in \eqref{eq:constraintMonotonicity} are satisfied. It remains to study the dependence of
	\begin{align*}
		\fin_j(\rho,t) 
			&	= mid\left\{d_i(\rho_i,t), s_j(\rho_j,t) - d_k(\rho_k,t), p_is_j(\rho_j,t)\right\}\\
			&	\qquad\qquad+mid\left\{d_k(\rho_k,t), s_j(\rho_j,t) - d_i(\rho_i,t), p_ks_j(\rho_j,t)\right\}
	\end{align*}
on $\rho_i$ and $\rho_k$. Since $d_i(\rho_i,t) > s_j(\rho_j,t) - d_k(\rho_k,t)$ and $d_k(\rho_k,t) > s_j(\rho_j,t) - d_i(\rho_i,t)$, we only have to study the following cases:
\begin{itemize}
	\item[a)] $s_j(\rho_j,t) - d_k(\rho_k,t) < d_i(\rho_i,t)\leq p_is_j(\rho_j,t)$: then by $p_i+p_k = 1$ we have $d_k(\rho_k,t)> p_ks_j(\rho_j,t)$. So
		\begin{itemize}
			\item[a.1)] if $p_ks_j(\rho_j,t)\leq s_j(\rho_j,t)-d_i(\rho_i,t) < d_k(\rho_k,t)$, then $\fin_j(\rho,t) = s_j(\rho_j,t)$;
			\item[a.2)] if $s_j(\rho_j,t)-d_i(\rho_i,t) \leq p_ks_j(\rho_j,t)< d_k(\rho_k,t)$, then $\fin_j(\rho,t) = p_ks_j(\rho_j,t) + d_i(\rho_i,t)$. Moreover, $d_i(\rho_i,t)\leq p_is_j(\rho_j,t)$ implies $p_ks_k(\rho_j,t) \leq s_k(\rho_j,t) - d_i(\rho_i,t)$, therefore actually $s_j(\rho_j,t)-d_i(\rho_i,t) = p_ks_j(\rho_j,t)$, hence $\fin_j(\rho,t) = s_j(\rho_j,t)$.
		\end{itemize}
	\item[b)] $s_j(\rho_j,t) - d_k(\rho_k,t) \leq p_is_j(\rho_j,t) \leq d_i(\rho_i,t)$: then again $p_i+p_k =1$ yields $p_ks_j(\rho_j,t) \geq s_j(\rho_j,t) - d_i(\rho_i,t)$, and then
		\begin{itemize}
			\item[b.1)] if $s_j(\rho_j,t)-d_i(\rho_i,t) \leq p_ks_j(\rho_j,t)\leq d_k(\rho_k,t)$, then $\fin_j(\rho,t) = s_j(\rho_j,t)$;
			\item[b.2)] if $s_j(\rho_j,t)-d_i(\rho_i,t) <d_k(\rho_k,t) \leq p_ks_j(\rho_j,t)$, then $\fin_j(\rho,t) = p_is_j(\rho_j,t) + d_k(\rho_k,t)$. Similarly to the point a.2), $d_k(\rho_k,t) \leq p_ks_j(\rho_j,t)$ implies $p_is_j(\rho_j,t) \leq s_j(\rho_j,t) - d_k(\rho_k,t)$, so that $p_is_j(\rho_j,t) = s_j(\rho_j,t) - d_k(\rho_k,t)$ and $\fin_j(\rho,t) = s_j(\rho_j,t)$.
		\end{itemize}
	\item[c)] $p_is_j(\rho_j,t) \leq s_j(\rho_j,t)-d_k(\rho_k,t) < d_i(\rho_i,t)$: then by $p_i+p_k =1$ it holds true $p_ks_j(\rho_j,t) \geq d_k(\rho_k,t)$, hence the only possibility is $s_j(\rho_j,t) - d_i(\rho_i,t) < d_k(\rho_k,t) \leq p_ks_j(\rho_j,t)$, which yields $\fin_j(\rho,t) = s_j(\rho_j,t)$.
\end{itemize}
In all cases it holds true $\fin_j(\rho,t) = s_j(\rho_j,t)$, so $\frac{\partial \fin_j(\rho,t)}{\partial \rho_i} =\frac{\partial \fin_j(\rho,t)}{\partial \rho_k} =0$, hence the leftmost inequalities in \eqref{eq:constraintMonotonicity} are also satisfied.
}

Basic properties characterizing monotone dynamical transportation networks are gathered in the following result.

\begin{theorem}
\label{theo:monotonicity+contraction}
Let \eqref{eq:systfirst} be a monotone dynamical transportation network. 
Let $\rho^{(1)}(t)$ and $\rho^{(2)}(t)$, for $t\ge0$, be solutions of \eqref{eq:systfirst} corresponding to initial conditions $\rho^{(1)}(0),\rho^{(2)}(0)\in\S$ with continuous inflow vectors $\lambda^{(1)}(t)$ and $\lambda^{(2)}(t)\geq0$ respectively.
Then, 
\begin{enumerate}
\item[(i)] if $\rho^{(1)}(0)\le\rho^{(2)}(0)$ and $\lambda^{(1)}(t)\le\lambda^{(2)}(t)$, then $\rho^{(1)}(t)\le\rho^{(2)}(t)$ for all $t\ge0$;
\item[(ii)] if $\lambda^{(1)}(t)=\lambda^{(2)}(t)$, then  $||\rho^{(1)}(t)-\rho^{(2)}(t)||_1$ is non-increasing in $t\ge0$.
\end{enumerate}
\end{theorem}
\begin{proof}
Notice that by \eqref{eq:constraintMonotonicity} and \eqref{eq:finfout} it holds true 
	\begin{align*}
		&\frac{\partial g_i(\rho,t)}{\partial \rho_j} \geq 0, &&\hspace{-2cm} \forall i\neq j\in\E	\\
		&\left.\frac{\partial g_i(\rho,t)}{\partial \lambda_k}\right|_{\lambda_k= \lambda_k(t)} \geq 0, &&\hspace{-2cm} \forall i\in\E\,, k\in\R
	\end{align*}

Point (i) is then a direct consequence of Kamke's theorem \cite[Theorem 1.2]{HirschPS03}, \cite{KamkeAM32} for monotone controlled systems \cite{AngeliTAC03}. 

Point (ii) is a consequence of Lemma~\ref{lem:l1-contraction} in Appendix. Indeed, $g(\cdot,\cdot)$ satisfies \eqref{equation:contractionAssumption1} by \eqref{eq:constraintMonotonicity}, and moreover
	$$
		\sum_{i\in\E}g_i(\rho,t) = \sum_{i\in\R}\lambda_i(t) - \sum_{i\in\R^o}d_i(\rho_i,t)
	$$
so \eqref{equation:contractionAssumption2} is also satisfied by the properties of the demand functions. Therefore, by \eqref{ineq:contraction}, along the evolution of the system it holds true
	$$
		\frac{d}{dt}\|\rho^{(1)}(t) - \rho^{(2)}(t)\|_1
		=
		\sum_{i\in\E}\sgn{\rho_i^{(1)}(t)-\rho_i^{(2)}(t)} \left( g_i(\rho^{(1)}(t),t) -g_i(\rho^{(1)}(t),t) \right)
		\leq 
		0
	$$
namely $\|\rho^{(1)}(t) - \rho^{(2)}(t)\|_1$ does not increase.
\end{proof}

We shall now investigate some consequences of the previous result.
\medskip

Point (i) of Theorem \ref{theo:monotonicity+contraction} states that the trajectories of a monotone dynamical transportation network are monotone systems, namely, they maintain the natural partial ordering with respect to initial conditions and inputs. An analogous monotonicity property holds for the solutions of some hyperbolic partial differential equations, including the celebrated Lighthill-Whitham and Richards model: e.g., cf.~Kru\v{z}kov's Theorem \cite[Proposition 2.3.6]{Serre:99} for entropy solutions of scalar conservation laws.

The monotonicity property established in point (i) of Theorem \ref{theo:monotonicity+contraction} implies the existence of equilibria, convergent solutions, and periodic solutions provided that the flow functions and the inflow vector are, respectively, constant, convergent, or periodic in time, as formalized in the following result. 

\begin{lemma}
\label{lemma:existenceEquilibriaPeriodicSolutions}
Let \eqref{eq:systfirst} be a monotone dynamical transportation network with continuous inflow vector $\lambda(t)\in\RR^\R$, $\lambda(t)\geq0$.
Then, 
\begin{enumerate}
\item[(i)] if the inflow vector $\lambda(t) \equiv \lambda^*$ and, for all $(i,j)$, the flow functions $f_{ij}(\rho,t)=f_{ij}^*(\rho)$ do not depend on time, and if for at least one finite initial condition, the corresponding trajectory does not grow unbounded in time, then there exists an equilibrium $\rho^*\in\S$ such that $\liminf\hat{\rho}(t)\ge\rho^*$ for every $\hat{\rho}(0)\in\S$. 
\item[(ii)] if the inflow vector and the flow functions are convergent in time, i.e., if $\lim_{t\to+\infty}\lambda(t)=\lambda^*$ and $\lim_{t\to+\infty}f_{ij}(\rho,t)=f^*_{ij}(\rho)$, for all $(i,j)$, and if for at least one finite initial condition, the corresponding trajectory does not grow unbounded in time, then there exists one trajectory such that $\rho(t)\to\rho^*$, and $\liminf\hat{\rho}(t)\ge\rho^*$ for every $\hat{\rho}(0)\in\S$.
\item[(iii)] if the inflow vector and the flow functions are periodic in time, i.e., if there exists some $T>0$ such that $\lambda(T+t)=\lambda(t)$ and $f_{ij}(\rho,t+T)=f_{ij}(\rho,t)$ for all $t\geq0$ and all $(i,j)$, and if for at least one finite initial condition the corresponding trajectory does not grow unbounded in time, then there exists a periodic solution $\rho(t+T)=\rho(t)$. 
\end{enumerate}
\end{lemma}

\begin{proof}
Let $\rho(t)$ be the trajectory of the system with zero initial condition $\rho(0) = \zerobf$. Under constant inflow vector and flow functions that do not depend on time, the properties of monotone systems ensure that $\rho(t)$ is increasing in time in every component. This implies that $\lim_{t\to\infty}\rho(t)$ exists. By point (ii) of Theorem~\ref{theo:monotonicity+contraction}, moreover, $\lim_{t\to\infty}\rho(t)$ is finite if and only if any trajectory with finite initial condition that does not grow unbounded in time. Under the assumptions, therefore, $\lim_{t\to\infty}\rho(t) = \rho^*$ is finite, so by Barbalat's lemma it is an equilibrium. Since by monotonicity for an arbitrary initial condition $\hat{\rho}(0)$ it holds $\hat{\rho}(0)\geq\zerobf$, we have $\hat{\rho}(t)\geq\rho(t)$ for all $t\geq0$, which proves (i).

Point (ii) is now an application of the converging input - converging state property of monotone controlled systems \cite{AngeliTAC03}.

Finally, in the periodic setting, recall that the system evolves according to $\dot{\rho}(t) = g(\rho(t),t)$ with $g(\rho,t+T) = g(\rho,t)$ for all $t\geq 0$ and $\rho\in\S$. Let $\phi(\hat{\rho}; t, t_0)$ be the evolution of the system starting at time $t_0$ with initial condition $\rho(t_0) = \hat{\rho}$ up to time $t\geq t_0$. We claim that $\phi(\hat{\rho}; t+kT, 0) = \phi(\phi(\hat{\rho}; kT,0), t, 0)$ for all $t\geq0$ and any nonnegative integer $k$. In fact, $y(t) = \phi(\hat{\rho}; t+kT, 0)$ and $x(t) = \phi(\phi(\hat{\rho}; kT,0), t, 0)$ are for $t\geq 0$ the solutions of 
	$$
		\begin{cases}
			\dot{y}(t) = g(y(t), t+kT)	= g(y(t), t)\\
			y(0) = \phi(\hat{\rho}; kT,0)
		\end{cases}
		\qquad
		\begin{cases}
			\dot{x}(t) = g(x(t), t)	\\
			x(0) = \phi(\hat{\rho}; kT,0)
		\end{cases}		
	$$	
By uniqueness of the solutions, $y(t) = x(t)$, that is, $\phi(\hat{\rho}; t+kT, 0) = \phi(\phi(\hat{\rho}; kT,0), t, 0)$, for all $t\geq0$. 

Consider now the map $F:\S\to \S$, $F(\hat{\rho}) = \phi(\hat{\rho}; T, 0)$, and set $F^k(\hat{\rho}) = F(F^{k-1}(\hat{\rho}))$. In this way we define the discrete time system $\rho(kT) = F^k(\rho(0))$, $\rho(0) = \hat{\rho}$. We claim that $F^k(\hat{\rho}) = \phi(\hat{\rho}; kT, 0)$. It holds true for $k=1$. For $k\geq2
$, by the previous argument and by induction, 
	$$
		F^k(\hat{\rho}) = F(F^{k-1}(\hat{\rho})) = F(\phi(\hat{\rho}; (k-1)T, 0)) =  \phi(\phi(\hat{\rho}; (k-1)T, 0); T,0) = \phi(\hat{\rho}; kT, 0)\,.
	$$
This immediately implies that the discrete time system is monotone, as $\hat{\rho}_1 \geq \hat{\rho}_2$ yields $F^k(\hat{\rho}_1) = \phi(\hat{\rho}_1; kT, 0) \geq \phi(\hat{\rho}_2; kT, 0) = F^k(\hat{\rho}_2)$. Let $\tilde{\rho}(kT)$ be the trajectory of the discrete time system with $\tilde{\rho}(0) = \zerobf$. Notice that $\tilde{\rho}(kT) = \phi(\zerobf; kT, 0)$, namely, $\tilde{\rho}(kT)$ is the sampled version of the trajectory in the original system with zero density initial conditions. By monotonicity, $\tilde{\rho}(kT)$ is increasing in each component and admits a limit $\lim_{k\to\infty}\tilde{\rho}(kT) = \rho^*_T$. Since $\tilde{\rho}(kT) = \phi(\zerobf; kT, 0)$, by point (ii) of Theorem~\ref{theo:monotonicity+contraction} if at least one trajectory of the original system does not grow unbounded, then $\rho^*_T$ is finite. 

To conclude, consider the trajectory $\rho(t) = \phi(\rho(0); t, 0)$ of the original system with initial condition $\rho(0) = \rho^*_T$. Since $\rho^*_T$ is a fixed point for $F$, we have $\rho(T) = F(\rho^*_T) = \rho^*_T = \rho(0)$, and therefore $\rho(t)$ is a periodic trajectory for the original system. 
\end{proof}
\medskip

In addition to point (i), point (ii) of Theorem~\ref{theo:monotonicity+contraction} states that a monotone dynamical transportation network is non-expansive in the $l_1$-distance, namely it is incrementally stable \cite{AngeliTAC:02}. In particular, Theorem \ref{theo:monotonicity+contraction} directly implies a general (weak) stability property: the distance from any reference trajectory $\rho^*(t)$, being it, e.g., an equilibrium, a periodic, or convergent solution, can never increase in time. However, in general it is not guaranteed that the $l_1$-distance between two trajectories is strictly decreasing. In contrast, it is possible to show \cite{Lovisari.Como.ea:CDC14} that there are cases, such as multiple equilibria, where the distance between two trajectories remains constant in time. In the following, we provide some sufficient conditions for the $l_1$-distance from a reference trajectory to be strictly decreasing. 

Before stating the following result, we introduce a state-dependent dual graph.
\begin{definition}
\label{def:dualGraph}
For every $\rho\in\S$ and $t\geq0$, where both $\fin_i(\rho,t)$ and $\fout_i(\rho,t)$ are differentiable for every $i\in\E$, we associate a directed dual graph $\H(\rho,t)$ with node set coinciding with the set of cells $\E$ and where there is a directed link from  $i$ to $j$ if and only if  
$\frac{\partial \fin_j}{\partial\rho_i}>0$ or $\frac{\partial \fout_j}{\partial\rho_i}<0$. 
We shall say that $\H(\rho,t)$ is rooted if, for all $i\in\E\setminus\R^o$, there is a directed path from $i$ to some $j\in\R^o$. 
\end{definition}

\begin{example}
\label{example:dualGraph}
Consider the network shown in Fig.~\ref{fig:ExampleDualGraph}, and the non-FIFO policy presented in Example~\ref{example:policiesNonFIFO}. For any $\rho$ and any time $t$ in which the network is in free-flow, the graph $\H(\rho,t)$ has a link $(e,j)$ if and only if $R_{ej} > 0$, i.e., $\tau_e = \sigma_j$. In other words, in free-flow, $\H(\rho,t)$ corresponds to a graph obtained by exchanging the roles of nodes and links of the original physical graph. Indeed, the dual graph associated with the matrix $J$ in the proof of Proposition~\ref{prop:freeflowGeneral} is $\H(\rho^*)$. If cell $4$ in Fig.~\ref{fig:ExampleDualGraph} is congested, namely, its inflow is bounded by its supply, then the directions of links $(3,4)$ and $(6,4)$ are reversed, and an additional link $(3,6)$ appears due to the interdependence of the dynamics of $3$ and $6$. Both graphs are rooted, since, for example, for every cell there is a directed path to at least one off-ramp.

\begin{figure}
	\centering
	\includegraphics[scale=0.6]{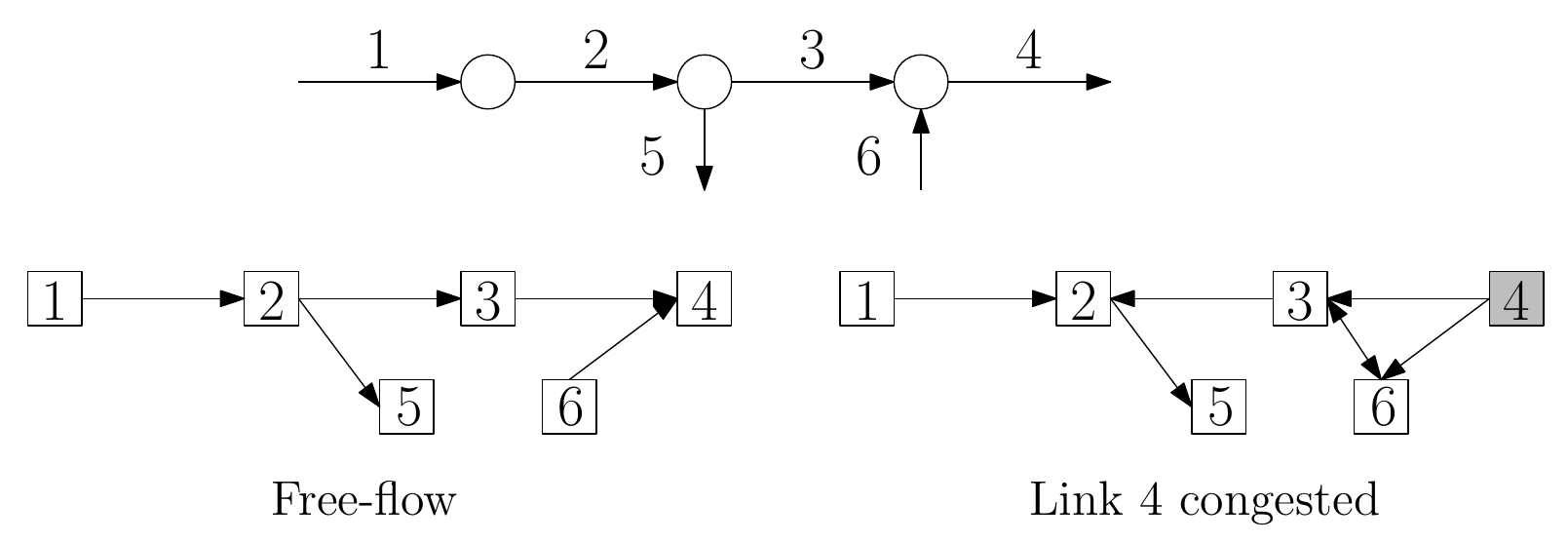} \\
	(a) \qquad \qquad \qquad \qquad \qquad (b)
	\caption{A network and the corresponding dual graphs in (a) free-flow and in (b) a congestion scenario.}
	\label{fig:ExampleDualGraph}
\end{figure}
\end{example}

We are now ready to state and prove the following Lemma, whose proof, being a bit technical, is postponed to Appendix \ref{technicalresults}.

\begin{lemma}\label{lemma:l1strictlydecreasing}
Let \eqref{eq:systfirst} be a monotone dynamical transportation network with continuous inflow vector $\lambda(t)\in\RR^\R$, $\lambda\geq0$.
 Let $\rho^{(1)}(t)$ and $\rho^{(2)}(t)$ be two solutions corresponding to initial conditions $\rho^{(1)}(0), \rho^{(2)}(0)\in\S$. Then, for all $t\ge0$ such that $\H(\rho^{(1)}(t),t)$ is rooted and $\rho^{(1)}(t) \neq \rho^{(2)}(t)$, one has that 
$$\frac{\de}{\de t}||\rho^{(1)}(t)-\rho^{(2)}(t)||_1<0\,.$$
\end{lemma}

Lemma \ref{lemma:l1strictlydecreasing} is instrumental in proving the next result, which provides sufficient conditions for reference trajectories, i.e., equilibria or periodic trajectories, to be globally asymptotically stable.

\begin{proposition}\label{prop:GAS}
Let \eqref{eq:systfirst} be a monotone dynamical transportation network with continuous inflow vector $\lambda(t)\in\RR^\R$, $\lambda\geq0$ 
such that $f_{ij}(\rho,t)$ are independent of $t$ for all $i\in\E\setminus \R^o$ and $j\in\E\setminus \R$. 
Then, 
\begin{enumerate}
\item[(i)] If $\rho^*$ is an equilibrium and $\H(\rho^*)$ is rooted, then $\rho^*$ is globally asymptotically stable; 
\item[(ii)] If $\rho^*(t)$ is a periodic solution and $\H(\rho^*(t),t)$ is rooted for some $t\ge0$, then $\rho^*(t)$ is a globally asymptotically stable periodic solution. 
\end{enumerate}
\end{proposition}
\begin{proof}
The two points are an immediate consequence of Lemma~\ref{lemma:l1strictlydecreasing}. In fact, assume $\rho^*$ is an equilibrium and $\H(\rho^*)$ is rooted. Let $\rho(t)$ be the trajectory of the system starting with an arbitrary initial condition $\rho_0$. Then 
	$$
		\frac{\de }{\de t}||\rho(t) - \rho^*||_1 < 0, \qquad \forall t
	$$
and thus $\rho(t) \stackrel{t\to \infty}{\longrightarrow}\rho^*$.

Similarly, let $\rho^*(t)$ be a periodic solution. If there exists $\hat{t}\in[0,T)$ such that $\H(\rho^*(\hat{t}+kT), \hat{t}+kT)$ is rooted for all integers $k\geq0$, then by continuity of the trajectory there exists $\delta > 0$ such that on each period $[kT, (k+1)T)$ the $l_1$-distance between $\rho^*(t)$ and $\rho(t)$ decreases by at least $\delta$, i.e., for all $k\geq 0$,
	$$
		||\rho((k+1)T) - \rho^*(t)||_1 \leq (1-\delta)||\rho(kT) - \rho^*(t)||_1\,,
	$$
so that $\mathrm{dist}(\rho(t), \rho^*(t)) \stackrel{t\to\infty}{\longrightarrow}0$. 

Since the initial condition was arbitrary, the previous argument establishes that the equilibrium or the periodic solution are globally attractive. Since stability is ensured by Theorem~\ref{theo:monotonicity+contraction}, the claim is proved.
\end{proof}

\begin{remark}
Notice that Proposition~\ref{prop:GAS} is \emph{not} limited to free-flow equilibria. In fact, even though for FIFO policies, the only equilibrium that a transportation network can admit is in free-flow, for other policies, such as those presented in Examples~\ref{example:policiesNonFIFO} and Example~\ref{example:mixture}, the network can admit equilibria in which some cells are congested, namely, their inflow is bounded by their supply. Examples are provided in Section~\ref{section:numericalExample}.
\end{remark}

Proposition~\ref{prop:GAS} can be used to generalize the property of local stability of free-flow equilibria in Proposition~\ref{prop:freeflowGeneral} to global stability, for monotone transportation networks, as stated in the next result.

\begin{theorem}
\label{theo:freeflow}
Let \eqref{eq:systfirst} be a monotone dynamical transportation network with time-invariant demand $d_i(\rho_i)$ and supply function $s_i(\rho_i)$ on every cell $i\in\E$, and constant turning preference matrix $R$ and inflow vector $\lambda\in\RR^{\R}_+$. Extend $\lambda$ to a non-negative vector in $\RR^\E$
by putting $\lambda_i=0$ for all $i\in\E\setminus\R$, and let $C\in\RR^\E$
be the vector of cells' flow capacity. Then, if \eqref{eq:conditionFreeflow} holds true, $\rho^*_i = d_i^{-1}(f^*_i)$, for all $i\in\E$, is a globally asymptotically stable equilibrium. 
\end{theorem}
\begin{proof}
The fact that $\rho^*$ is an equilibrium was proven in the proof of Proposition~\ref{prop:freeflowGeneral}. By the same proof, and as in Example~\ref{example:dualGraph}, it is straightforward to see that the connectivity assumption on the actual physical graph implies that $\H(\rho^*)$ is rooted. Then $\rho^*$ is globally asymptotically stable by Proposition~\ref{prop:GAS}.
\end{proof}

\begin{remark}
\label{remark:nonmonotone}
While Example~\ref{example:counterexampleGASFIFO} shows that, for FIFO policies, the free-flow equilibrium is not globally asymptotically stable in general, 
Theorem~\ref{theo:freeflow} can be generalized to non-monotone policies in some special cases. The key is that, while monotonicity is a sufficient condition for the $\ell_1$ contraction principle in Lemma~\ref{lemma:l1strictlydecreasing} to hold true, it is not necessary. This happens, e.g., when every node in the network is either a merge, or all its outgoing cells are off-ramps. For such networks, it can be shown that the free-flow equilibrium from Theorem~\ref{theo:freeflow} is globally asymptotically stable under the mixture model for sufficiently small $\theta \in (0,1]$.
\end{remark}

\section{Control of Dynamical Transportation Networks}
\label{sec:control}

In this section, we describe how to cast the equilibrium selection and the optimal control of dynamical transportation networks as convex optimizations or linear programs. Throughout, we will consider \emph{uncontrolled} demand and supply functions, 
	$
		d^u_i(\rho_i,t)\,, s_i^u(\rho_i,t)\,, i\in\E\,,
	$
which are both \emph{concave} and, respectively, strictly increasing and non-increasing in $\rho_i$. 
We will measure the system performance through cost functions $\Psi:\,\RR_+^{\E} \to \RR_+$ that are \emph{convex, and strictly increasing} in each component. A relevant example is provided by the weighted sum of cell-wise densities 
	\begin{equation}
		\label{Psilinear}
		\Psi(\rho) = \sum_{i\in\E}\eta_i\rho_i\,,
	\end{equation}
for non-negative weights $\eta_i\geq0$, $i\in\E$, which recovers standard performance metrics such as the Total Travel Time, e.g., see  \cite{GomesTRC06}. 
We will assume that the \emph{controlled} demand functions have the following form 
	$$
		d_i(\rho_i,t) = \alpha_i(t)d_i^u(\rho_i,t)	\,,\qquad i\in\E\,,
	$$
where $\alpha_i(t) \in [0,1]$ are control parameters. In the context of freeway networks, a given $\alpha_i(t)$ can be realized through appropriate setting of speed limits. In particular, if the uncontrolled demand function on cell $i$ is linear as in \eqref{eq:linearDemand}, then its rescaling is equivalent to the modulation of the free-flow speed $v_i(t)=v_i^{u}(t)\alpha_i(t)$ \cite{HegyiTRC05}, where $v_i^{u}(t)$ could be interpreted as the maximum possible speed due to, e.g., safety considerations. 

We will consider two distinct settings combining the control of demand functions described above with 
	\begin{description}
		\item[(I)] \emph{control} of the turning preference matrix; or 
	\end{description}
	\begin{description}
		\item[(II)] supply control. 
	\end{description}
	
By \textbf{(I)}, we mean the capability of modifying an uncontrolled turning preference matrix $R^u(t)$ to a controlled one $R(t)$ which still has nonnegative entries and row sums equal to $1$, and satisfies the additional constraint $d_i(\rho_i,t)R_{ij}(t)\le d_i^u(\rho_i,t)R_{ij}^u(t)$. In other words, demand control combined with control of the turning preference matrix amounts to the ability of independently reducing the demand from cell $i$ that intends to turn to cell $j$. On the other hand, by \textbf{(II)} we refer to the possibility of saturating the supply functions 
	$$
		s_i(\rho_i,t) = \min\{s_i^u(\rho_i,t),\beta_i(t)\}	\,,\qquad i\in\E\,,
	$$
where $s_i^u(\rho_i,t)$, $i \in \E$ are the uncontrolled supply functions, and $\beta_i(t)\ge0$, $i \in \E$ are control parameters to be actuated, e.g., through metering. 

We first present results on the optimal equilibrium selection for dynamical transportation networks  where the inflows, the uncontrolled supply and demand functions and turning preference matrix are all time-invariant. Then, we will deal with the optimal control problem for dynamical transportation networks with general time-varying parameters. 

\subsection{Equilibrium selection}
\label{subsec:equilibriumcontrol}

We start by characterizing the set of all possible equilibria associated to the stationary case, where the uncontrolled supply and demand functions, as well as the inflow vector and the turning preference matrix are time-invariant. 
Consider the set $\F\subseteq\RR_+^{\E}\times\RR_+^{\E\times\E}$ of pairs $(x,y)$ of a density vector $x$ and a cell-to-cell flow matrix $y$ satisfying the following constraints 
	\begin{equation}
		\label{Fdef} 
		\begin{array}{rcl}
			\ds\sum_iy_{ij}\le s_j^u(x_j)			&	\qquad		&\forall j\in\E \\[10pt] 
			y_{ij}\le R^u_{ij}d_i^u(x_i)					&	\qquad		&\forall i, j\in\E \\[10pt] 
			\ds\sum_iy_{ij}=\ds\sum_iy_{ji}		&	\qquad		&\forall j\in\E \setminus(\R \cup\R^o)\\[10pt] 
			\ds\lambda_i=\sum_jy_{ij}					&	\qquad		&\forall i\in\R \\[10pt]
			\ds\sum_j y_{ji}\le d_i^u(x_i)				&	\qquad		&\forall i\in\R^o
		\end{array}
	\end{equation}
The following result guarantees that every equilibrium of the dynamical transportation network belongs to the set $\F$. 
\begin{lemma}
\label{lemma:FcontainsAllEquilibria}
Consider a dynamical transportation network where the uncontrolled supply and demand functions, as well as the inflow vector and the uncontrolled turning preference matrix are time-invariant. If $\rho^*\in\S$ is an equilibrium, then $(\rho^*,f^*)\in\F$, where $f^*=\{f_{ij}(\rho^*)\}_{i,j\in\E}$.
\end{lemma}
\begin{proof}
The conservation law \eqref{eq:systfirst} along with the definition of inflow and outflow \eqref{eq:finfout} ensures that $f^*$ satisfies the last three constraints in \eqref{Fdef}. In particular, the last set of constraints is satisfied with equality since at equilibrium $\sum_{j\in\E_i^-}f_{ji}(\rho^*) = \fin_i(\rho^*) = \fout_i(\rho^*) = d_i(\rho_i^*)$ for all offramps $i\in\R^o$. Finally, $f^*$ satisfies the first two constraints in \eqref{Fdef} because of \eqref{eq:constraintFin} and \eqref{eq:constraintFetoj}.
\end{proof}

Observe that the constraints that characterize $\F$ are convex. 
This implies that the optimization 
	\begin{equation}
	\label{optimization}
		\min_{(x,y)\in\F}\Psi(x)
	\end{equation}
is a convex problem. Moreover, in the case of linear demand, affine supply, and linear cost function \eqref{Psilinear}, the optimization \eqref{optimization} is a linear program. Convexity and linearity are extremely appealing properties of optimization problems, as convex and linear programs are classes of problems for which efficient algorithms, solvers, and toolboxes have been developed and tested. Indeed, in our simulations in Section~\ref{section:numericalExample} we use the \verb+Matlab+ package \verb+CVX+ \cite{CVXSoftware, GrantRALC08}. A promising next step is to adapt the Alternating Direction Method of Multipliers (ADMM) \cite{BoydFTML11} to the problem under analysis. ADMM is a popular approach to solve optimization problems on networks in a \emph{distributed} fashion, namely, the algorithm relies on a network of agents which perform local computations and exchange information with nearest neighbors to solve the optimization problem in an iterative manner. The computational complexity of ADMM scales nicely with the size of the network too. We intend to pursue this direction in future research. Concluding, the availability of off-the-shelf algorithms and solvers is the reason why we are particularly interested in the case in which \eqref{optimization} is convex. However, notice that the proposed control strategies would not change if convexity were lost, e.g., because the supply functions are not concave. The only difference is that solving \eqref{optimization} in the non-convex case would be computationally hard.
 
We now address the question of how to design control parameters such that the solution of the optimization \eqref{optimization} is a (stable) equilibrium for the controlled dynamical transportation network. We first consider case \textbf{(I)} where the demand control is combined with control of the turning preference matrix. 

\begin{proposition}
\label{proposition:turningPreferenceControl}
Consider a dynamical transportation network where the uncontrolled demand functions $d_i^u(\cdot)$ and supply functions $s_i^u(\cdot)$, as well as the inflow vector $\lambda$ and the uncontrolled turning preference matrix $R^u$ are all time-invariant. Let $(x^*,y^*)$ be a solution of the optimization \eqref{optimization}. Set time-invariant demand controls $\alpha_i$, controlled turning preference matrix $R$, and supply control $\beta_i$ as follows
	\begin{equation}
		\label{eq:controlVariablesFullControl}
		\begin{array}{rll}
			\nonumber\alpha_{i} &	= 
				\begin{cases}
				\frac{\sum_{k\in\E_i^+}y_{ik}^*}{d_i^u(x_i^*)},&	\textrm{ if }x_i^*\neq 0	\\
				0, &\textrm{ if }x_i^*= 0 
				\end{cases}
				&\qquad \forall i \in \E \setminus \R^o	\\[5pt] 
			R_{ij} &= 
				\begin{cases}
					\frac{y_{ij}^*}{\sum_{k\in\E_i^+}y_{ik}^*},& \textrm{ if } \sum_{k\in\E_i^+}y_{ik}^*\neq 0	\\
					\frac{1}{|\E_i^+|}, & \textrm{ if }\sum_{k\in\E_i^+}y_{ik}^*= 0	\\
				\end{cases}
				&\qquad \forall i, j \in \E \setminus \R^o, \, i\neq j	\\[5pt] 
		\nonumber\beta_i &= +\infty &\qquad \forall i \in \E \setminus \R \, 
		\end{array}
	\end{equation}
Then $\alpha_iR_{ij}\le R^u_{ij}$  for all $i, j \in \E$, and $x^*$ is a stable free-flow equilibrium for the controlled dynamical transportation network. 
Moreover, if $\sum_{j\in\E_i^-}y_{ji}^* < \fmax_i$ for all $i\in\E$, then $x^*$ is locally asymptotically stable, and if, in addition, the dynamical transportation network is monotone, then $x^*$ is globally asymptotically stable. 
\end{proposition}
\begin{proof}
Let $(x^*, y^*)\in\argmin_{(x,y)\in\F}\Psi(x)$ be a solution of the optimization in \eqref{optimization} and set the control signals as in \eqref{eq:controlVariablesFullControl}. Notice that $\beta_i=+\infty$ implies that no supply control is used. The choice of control parameters implies that $y_{ij}^*=\alpha_i R_{ij} d_i^u(x_i^*)=R_{ij} d_i(x_i^*)$ for all $i, j \in \E \setminus \R^o$, $i \neq j$. Then, for all $j \in \E \setminus \R$,
	\begin{align*}
		\sum_{i\in\E_j^-}R_{ij}d_i(x_i^*) 	&	= \sum_{i\in\E_j^-}\alpha_i R_{ij}d_i^u(x_i^*)
																					 = \sum_{i\in\E_j^-}y_{ij}^* \leq s_j^u(x_j^*)
	\end{align*}
where the last inequality is implied by the first constraint in \eqref{Fdef}.

Since this holds for all $i$, \eqref{eq:constraintFetojEquality} guarantees that $f_{ij}(x^*) = R_{ij}d_i(x_i^*)$, and hence $f_{ij}(x^*) = y_{ij}^*$, for all $i,j$. Then, the third and fourth constraints in \eqref{Fdef} imply that 
	\begin{equation}	
		\label{fin=fout} 
		\fin_i(x^*)=\fout_i(x^*)\,,
	\end{equation}
for every cell $i\in\E\setminus\R^o$. For off-ramps, \eqref{fin=fout} follows from the fact the last constraint in \eqref{Fdef} is necessarily satisfied with equality, for otherwise, if $\sum_{k}y_{ki}^*<d_i^u(x_i^*)$ for some $i\in\R^o$, a small decrease in $x^*_i$ would reduce the value of the objective function without violating any constraints. 

Finally, since $f_{ji}(x^*) = \R_{ji}^ud_j(x_j^*)$ for all $(j,i)$, $x^*$ is an equilibrium in free-flow. To conclude, let $\sum_{j\in\E_i^-}y_{ji}^* < \fmax_i$ for all $i\in\E$. 
Local stability of $x^*$ follows from Proposition~\ref{prop:freeflowGeneral}, and the global asymptotic stability, for monotone networks, follows from Theorem~\ref{theo:freeflow}.
\end{proof}

We now turn our attention to case \textbf{(II)}, where the turning preference matrix cannot be controlled, but, in addition to demand functions, supply functions can also be controlled. Our main result, stated below, is restricted to the case of monotone dynamical transportation networks where each node is either a merge, i.e., it has one outgoing cell, or a diverge, i.e., it has one incoming cell and multiple outgoing cells. Note that a node with one incoming and one outgoing cell will be referred to as a merge node.

\begin{proposition}
\label{proposition:SpeedLimitAndSupplyControl}
Consider a monotone dynamical transportation network where the uncontrolled demand functions $d_i^u(\cdot)$ and supply functions $s_i^u(\cdot)$, as well as the inflow vector $\lambda$ and the uncontrolled turning preference matrix $R^u$ are all time-invariant. Assume that each node is either a merge or a diverge. Let $(x^*,y^*)$ be a solution of the optimization \eqref{optimization}. Set time-invariant demand controls $\alpha_i$, supply controls $\beta_i$, and matrix of turning preferences $R$ as follows
\begin{equation}
		\label{eq:controlVariablesSpeedLimitAndSupplySaturation}
		\begin{array}{rll}
		\alpha_{i} &	= 
			\begin{cases}
				\frac{y_{ij}^*}{d_i^u(x_i^*)}, & \textrm{ if } x_i^* > 0, \tau_i \textrm{ is a merge}, \{j\} = \E_i^+\\
				0,& \textrm{ if } x_i^* = 0, \tau_i \textrm{ is a merge}	\\
				1, & \textrm{ if } \tau_i \textrm{ is a diverge}
			\end{cases}
				&\qquad \forall i \in \E \setminus \R^o	\\[5pt] 
			R_{ij} &= R_{ij}^u
				&\qquad \forall i, j \in \E \setminus \R^o, \, i\neq j	\\[5pt] 
		\nonumber\beta_i &= \begin{cases}
				y^*_{ji}, &\sigma_i \textrm{ is a diverge }, \{j\} = \E_i^-	\\
				+ \infty, & \sigma_i \textrm{ is a merge}
			\end{cases} &\qquad \forall i \in \E \setminus \R \, 
		\end{array}
	\end{equation}
Then, $x^*$ is a stable equilibrium for the controlled dynamical transportation network. 
\end{proposition}
\begin{proof}
Let $(x^*, y^*)\in\argmin_{(x,y)\in\F}\Psi(x)$ be a solution of the optimization in \eqref{optimization}. Set the control parameters as in \eqref{eq:controlVariablesSpeedLimitAndSupplySaturation}, and notice that the turning preferences are not modified by the present control strategy, and $\beta_i=+\infty$ implies that no supply control is used.

As in the proof of Proposition~\ref{proposition:turningPreferenceControl}, we shall prove that $f_{ij}(x^*) = y_{ij}^*$ for all $i,j$. This in turn implies that $x^*$ is an equilibrium for the system. Stability is ensured by point (ii) of Theorem~\ref{theo:monotonicity+contraction}.

To this aim, let $v$ be a merge node with $\{j\} = \E_v^+$. If $\alpha_i$ is computed according to \eqref{eq:controlVariablesSpeedLimitAndSupplySaturation}, then (notice  $R_{ij} = 1$ for all $i\in\E_v^-$)
	$$
		\sum_{i\in\E_v^-}d_i(x_i^*) = \sum_{i\in\E_v^-}\alpha_id_i^u(x_i^*) = \sum_{i\in\E_j^-}y_{ij}^* \leq s_j^u(x_j^*)\,
	$$
where the last inequality is implied by the first constraint in \eqref{Fdef}.	
Therefore, \eqref{eq:constraintFetojEquality} guarantees that $f_{ij}(x^*) = d_i(x_i^*) = \alpha_id_i^u(x_i^*) = y_{ij}^*$, for all $i\in\E_v^-$.

Let instead $v$ be a diverge node with $\{i\} = \E_v^-$. First of all, we claim that the monotonicity conditions in \eqref{eq:constraintMonotonicity} imply that $f_{ij}(x) = \min\{R_{ij}d_i(x_i), s_j(x_j)\}$ for any $x\in\S$. We study two cases:
\begin{itemize}
	\item $R_{ij} d_i(x_i) \leq s_i(x_j)$: by \eqref{eq:constraintFetoj}, $f_{ij}(x) \leq R_{ij}d_i(x_i)$. If $f_{ij}(x) = R_{ij} d_i(x_i)$, the claim is proved, so assume by contradiction $f_{ij}(x) < R_{ij} d_i(x_i)$. Let $\hat{x}$ such that $\hat{x}_i = x_i$, $\hat{x}_j = x_j$, and $\hat{x}_k \leq x_k$, $k\in\E_v^+$, $j\neq k$, such that $R_{ik} d_i(x_i) = R_{ik} d_i(\hat{x}_i) \leq s_k(\hat{x}_k)$, for all $k\in\E_v^+$. Then \eqref{eq:constraintFetojEquality} implies $f_{ij}(\hat{x}) = R_{ij} d_i(x_i)$. Let $\gamma := \{\theta x + (1-\theta)\hat{x}:\,\theta\in[0,1]\}$ be a path from $\hat{x}$ to $x$. Then
	$$
		f_{ij}(x) = f_{ij}(\hat{x}) + \int_\gamma \nabla f_{ij}(\xi)d\xi \geq f_{ij}(\hat{x})
	$$
where the inequality follows by monotonicity since, the components of the state changing (increasing) along $\gamma$ do not include $j$. Thus, $f_{ij}(x) \geq f_{ij}(\hat{x}) = R_{ij}d_i(x_i)$ and $f_{ij}(x) < R_{ij} d_i(x_i)$, a contradiction. Therefore, if $R_{ij} d_i(x_i) \leq s_i(x_j)$, then $f_{ij}(x) = R_{ij} d_i(x_i)$. 
	\item $R_{ij} d_i(x_i) > s_j(x_j)$: by \eqref{eq:constraintFin}, $f_{ij}(x) \leq s_j(x_j)$. If $f_{ij}(x) = s_j(x_j)$, the claim is proved, so assume by contradiction $f_{ij}(x) < s_j(x_j)$. Let $\hat{x}$ be such that $\hat{x}_k = x_k$ for all $k\in\E_v^+$ and $\hat{x}_i < x_i$ be such that $R_{ij} d_i(\hat{x}_i) = s_j(\hat{x}_j) = s_j(x_j)$. Then, making explicit the dependence of $f_{ij}(x)$ on the $i$-th component of the state by writing $f_{ij}(x) = f_{ij}(x_i, \{x_k\}_{k\in\E_v^+})$, we obtain
	$$
		f_{ij}(x) = f_{ij}(\hat{x}) + \int_{\hat{x}_i}^{x_i}\frac{\partial}{\partial \xi_i}f_{ij}(\xi_i, \{x_k\}_{k\in\E_v^+})d\xi_i \geq f_{ij}(\hat{x}) = s_j(\hat{x}_j) = s_j(x_j)
	$$		
where the inequality holds again by monotonicity since $\frac{\partial f_{ij}(x)}{\partial x_i}\geq 0$. Thus $f_{ij}(x) < s_j(x_j)$ and $f_{ij}(x) \geq s_j(x_j)$, once again a contradiction. Therefore, if $R_{ij} d_i(x_i) > s_j(x_j)$, then $f_{ij}(x) = s_j(x_j)$.
\end{itemize}

In conclusion, $f_{ij}(x) = \min\{R_{ij}d_i(x_i), s_j(x_j)\}$ for all $x$, and thus in particular for $x^*$.

For the supply control computed according to \eqref{eq:controlVariablesSpeedLimitAndSupplySaturation}, we have (recall $R_{ij} = R_{ij}^u$ for all $(i,j)$ and $\alpha_i =1$ since $\tau_i = v$ is a diverge, so that $d_i(\cdot) = d_i^u(\cdot)$)
	$$
		f_{ij}(x^*) = \min\{R_{ij}d_i(x_i^*), s_j(x_j^*)\} = \min\{R_{ji}^ud_i^u(x_i^*), s_j^u(x_j^*), y^*_{ij}\}
	$$
Then, the constraints in \eqref{Fdef} yield $y^*_{ij} \leq R_{ij}^ud_i^u(x_i^*)$ and $y^*_{ij} \leq s_j^u(x_j^*)$, thus establishing that $f_{ij}(x^*) = y^*_{ij}$, for all $j\in\E_v^+$. 

\end{proof}

We end this subsection by considering a restriction of case \textbf{(II)}, where the demand and supply functions cannot be controlled over a subset of cells $\E^u$, i.e., $\alpha_i \equiv 1$ and $\beta_i \equiv + \infty$ if $i \in \E^u$. For simplicity, we again focus on the case where each node is either a merge or a diverge. In this case, the equilibrium selection problem under partial control can be written as 
	\begin{align}
	\nonumber						&	\min_{(x,y)\in\F}		&	\Psi(x)\\
	\label{optimizationSpatConst}		&	\quad \mathrm{s.t.}			&	y_{ij} = s_j^u(x_j), \, \,  	& 	\forall j\in\E^u, \,  \{i\} = \E_j^-, \, \, \sigma_j\textrm{ is diverge}\\
	\nonumber						& 								&	y_{ij} = d_i^u(\rho_i), \, \, &	\forall i\in\E^u, \, \{j\} = \E_i^+, \, \tau_i \textrm{ is merge}
	\end{align}

Observe that the feasible set of \eqref{optimizationSpatConst} is a subset of the feasible set of \eqref{optimization}. Therefore, in general, a solution of \eqref{optimizationSpatConst} will have a greater cost in comparison to the solution of \eqref{optimization}. 
The possibility of implementing the optimal solution $(x^*,y^*)$ of the original optimization in \eqref{optimization} using partial control is left to future study.
The constraints of the type $y_{ij} = s_j^u(x_j)$ and $y_{ij} = d_i^u(x_i)$ are convex only if $s_j^u(\cdot)$ and $d_i^u(\rho_i)$ are affine. This condition is satisfied for the standard linear demand and affine supply functions as in \eqref{eq:linearDemand} and \eqref{eq:affinesupply} respectively. The following result is the analogous of Proposition~\ref{proposition:SpeedLimitAndSupplyControl} for the partial control case. 
The proof, which relies on setting the control signals $\alpha$ and $\beta$ as in \eqref{eq:controlVariablesSpeedLimitAndSupplySaturation}, is omitted. 

\begin{proposition}
\label{proposition:SpeedLimitAndSupplySpatiallyConstrainedControl}
Consider a monotone dynamical transportation network where the demand functions $d_i^u(\cdot)$ and supply functions $s_i^u(\cdot)$, as well as the inflow vector $\lambda$ and the uncontrolled turning preference matrix $R^u$ are all time-invariant. In addition, assume that each node is either a merge or a diverge, and that demand and supply functions on uncontrolled cells are affine. Let $(x^*,y^*)$ be an optimal solution of \eqref{optimizationSpatConst}. Set time-invariant demand controls $\alpha_i$, supply controls $\beta_i$, and matrix of turning preferences $R$, as in \eqref{eq:controlVariablesSpeedLimitAndSupplySaturation}. Then $x^*$ is a stable equilibrium for the controlled dynamical transportation network.
\end{proposition}

\subsection{Optimal control}
\label{subsec:timevaryingcontrol}

In this subsection, we study the problem of optimal control for dynamical transportation networks. We consider the general case where uncontrolled supply functions $s_i^u(\rho_i,t)$ and demand functions $d_i^u(\rho_i,t)$, as well as the inflow vector $\lambda(t)$ and the turning preference matrix $R(t)$ are Lipschitz continuous functions of time. Throughout, we shall discuss the control strategy corresponding to case \textbf{(I)}, where we allow control of turning preference matrix and speed limits. The strategy corresponding to case \textbf{(II)} of controlling speed limits and supply function, can be developed in a totally analogous way, and is therefore omitted.

The optimal control framework of this subsection can be used as a basis for model predictive control strategy for dynamical transportation networks, e.g., see \cite{HegyiTRC05}, as follows. The network state $\rho_0=\rho(t_0)$ is observed at some initial time $t_0$ and the future arrival rate $\lambda(t)$ over some interval $[t_0,t_0+H]$ is estimated, possibly using historical information. It is desired to compute control actions which can be applied in an open loop fashion from $t_0$ to $t_0+H$ (or earlier), at which point new observations and estimations are made and the process is repeated. Notice that, in standard model predictive control (MPC), the control is only applied in the interval $[t_0, t_0+H')$, for some $H'<H$, and then it is recomputed for the next horizon $[t_0+H', t_0+H'+H]$. For the sake of simplicity, and without loss of generality, in this paper we consider the case when $H' = H$. 
Analogous to \eqref{Fdef}, let $\F_H(t_0,\rho_0)$ be the set of triple $(x(t),y(t),t)\in\RR_+^{\E}\times\RR_+^{\E\times\E}\times[0,H)$ that are continuous in $t$ and satisfy the following constraints 
	\begin{equation}
		\label{FdefHt0} 
		\begin{array}{rcl}
			\ds\sum_iy_{ij}(t)\le s_j^u(x_j(t),t)				&		&\forall j\in\E, t\in[t_0,t_0+H) \\[5pt] 
			y_{ij}(t)\le R_{ij}^u(t)d_i^u(x_i(t),t)				&		&\forall i\ne j\in\E, t\in[t_0,t_0+H) \\[5pt] 
			\ds\dot{x}_j = \sum_iy_{ij}-\ds\sum_iy_{ji}			&		&\forall j\in\E \setminus(\R \cup\R^o), 
				 t\in[t_0,t_0+H)\\[5pt] 
			\ds\dot{x}_j = \lambda_j-\sum_iy_{ji}				&		&\forall j\in\R, t\in[t_0,t_0+H) \\[5pt]
			\ds\dot{x}_j \leq \sum_i y_{ij} - d_j^u(x_j,t)		&		&\forall j\in\R^o, t\in[t_0,t_0+H)\\
			x(t_0) = \rho_0
		\end{array}
	\end{equation}
	
We consider the following optimal control problem:
	\begin{equation}
	\label{optimizationTimeVarying}
		\min_{\left(x(t),y(t),t\right)\in \F_{H}(t_0,\rho_0)}\int_{t_0}^{t_0+H}\Psi(x(s),s) \, ds
	\end{equation}
where $\Psi(\cdot,t)$ is convex and strictly increasing in each component for every $t\in[t_0, t_0+H]$. The cost function in \eqref{optimizationTimeVarying} can be chosen to penalize the transient as well as the terminal state. Possible examples are $\Psi(x(s),s) = \sum_eL_ex_e(s)$, where $L_e$ is the length of cell $e$,  representing the total volume of vehicles in the network, or, as in evacuation problems, $\Psi(x(s),s) = -\sum_{e\in\R^o}d_e(x_e(s))$ \cite{LiTCNS:14}, aiming to maximize the total outflow from the network in the given time horizon.

\begin{remark}
Problem \ref{optimizationTimeVarying} is a convex problem in a set of continuous functions, as can be easily seen using concavity of demand and supply functions, and linearity of the derivative operator. However, since the variables are taken from an infinite dimensional space, its solution is not as straightforward as in the stationary case. A simple strategy, which we use in our simulations in Section~\ref{section:numericalExample}, is to discretize the time with a small enough step size $\Delta t$. A first order discretization can be easily seen to maintain the convexity properties because expressions such as $\dot{x}_j$ are replaced with $\frac{x_j(t+\Delta t) - x_j(t)}{\Delta t}$. After discretization, the variables of the problem are $x_i(k\Delta t)$ for all $i\in\E$ and $y_{ij}(k\Delta t)$ for all $(i,j)$, and all integer $k \in [0, \frac{H}{\Delta t}-1]$. 
\end{remark}

The following results state that $\F_H(t_0,\rho_0)$ contains all the trajectories starting at $t_0$ with initial condition $\rho_0$, and conversely that there exist control signals such that a solution $\{(x^*(t), y^*(t),t): \, t\in[t_0,t_0+H)\}$ of \eqref{optimizationTimeVarying} is a feasible trajectory for the controlled dynamical transportation network. The proof is very similar to the proof of Lemma~\ref{lemma:FcontainsAllEquilibria} and Proposition~\ref{proposition:turningPreferenceControl}, and is therefore omitted. In particular, the control signals in Proposition~\ref{proposition:turningPreferenceControlTrajectories} are to be set according to the following, for all $t\in[t_0,t_0+H)$:
	\begin{equation}
		\label{eq:controlVariablesFullControlTrajectory}
		\begin{array}{rll}
			\nonumber\alpha_{i}(t) &	= 
				\begin{cases}
					\frac{\sum_{k\in\E_i^+}y_{ik}^*(t)}{d_i^u(x_i^*(t))},&	\textrm{ if } x_i^*(t) \neq 0	\\
					0,&\textrm{ if } x_i^*(t) = 0
				\end{cases}
				 &\qquad \forall i \in \E \setminus \R^o 	\\[5pt] 
			R_{ij}(t) &= 
				\begin{cases}
					\frac{y_{ij}^*(t)}{\sum_{k\in\E_i^+}y_{ik}^*(t)}, &\textrm{ if } \sum_{k\in\E_i^+}y_{ik}^*(t) \neq 0	\\
					0, & \textrm{ if } \sum_{k\in\E_i^+}y_{ik}^*(t) = 0
				\end{cases}
				&\qquad \forall i, j \in \E \setminus \R^o, \, 
			i\neq j	\\[5pt] 
		\nonumber\beta_i(t) & \equiv + \infty, &\qquad \forall i \in \E \setminus \R \, .
		\end{array}
	\end{equation}

\begin{lemma}
\label{lemma:FcontinsAllTrajectories}
Consider a dynamical transportation network where the uncontrolled demand and supply, as well as the inflow vector and the uncontrolled turning preference matrix are Lipschitz continuous functions of the time. If $\{\rho^*(t): \, t\in[t_0,t_0+H)\}$ is a trajectory of the system with initial condition $\rho^*(t_0) = \rho_0$, then $\{(\rho^*(t), f^*(t),t): \, t\in[t_0,t_0+H)\} \in \F_H(t_0,\rho_0)$, where $f^*(t)=\{f_{ij}(\rho^*(t))\}_{i,j\in\E}$.
\end{lemma}

\begin{proposition}
\label{proposition:turningPreferenceControlTrajectories}
Consider a dynamical transportation network where the uncontrolled demand and supply, as well as the inflow vector and the uncontrolled turning preference matrix are Lipschitz continuous functions of the time. Let $\{(x^*(t),y^*(t),t): \, t\in[t_0,t_0+H)\}$ be a solution of the optimization \eqref{optimizationTimeVarying}. Set the piecewise continuous time-varying demand controls $\alpha_i(t)$, controlled turning preference matrix $R(t)$, and supply controls $\beta_i(t)$
, as in \eqref{eq:controlVariablesFullControlTrajectory}. Then $\alpha_i(t)R_{ij}(t)\le R^u_{ij}(t)$  for all $i,j$, and $\{x^*(t): \, t\in[t_0,t_0+H)\}$ is a trajectory in free-flow for the controlled dynamical transportation network. 
\end{proposition}

Finally, optimal control over a fixed time horizon can be easily recast as periodic trajectory selection. Indeed, let inflows and turning preference matrix be periodic of period $T$, i.e., $\lambda(t) = \lambda(t+T)$ and $R(t+T) = R(t)$ for all $t\geq 0$. In this case, let $\F_T$ be given as in \eqref{FdefHt0} with $t_0 = 0$, $H = T$ and where the last equality constraint is replaced with $x(0) = x(T)$. The periodic trajectory selection problem can be formally defined in the same way as \eqref{optimizationTimeVarying}, and a solution is again  provided by Proposition~\ref{proposition:turningPreferenceControlTrajectories}, setting $t_0 = 0$ and $H = T$. In this case, the solution is a periodic trajectory for the controlled system with period $T$, whose stability can be studied using the tools provided by Proposition~\ref{prop:GAS}. The difference between the two cases is that while optimal control is a on-line feedback strategy that relies on measurements of the actual state of the network to compute the controls, the periodic trajectory selection problem can be solved off-line and the corresponding controls can be applied in open-loop. As such, it can be seen as an appealing solution to find optimal controls on the basis of periodic daily or weekly data, as an alternative to standard strategies in which each day is partitioned into different time periods, such as the classical night-commuting-afternoon-commuting cycle, for each of which static controls are computed.

\section{Simulation Studies}
\label{section:numericalExample}

In this section, we illustrate the theoretical findings of Sections~\ref{sec:stability} and \ref{sec:control} with simulation studies performed on a transportation network loosely inspired by the freeway system in the southern region of Los Angeles. In particular, the network that we chose consists of the state routes 91 and 46, and interstate highways I-110, I-710 and I-405, as shown in Fig.~\ref{fig:Map}. Along with the main lines of the freeway system, the network consists of several on- and off-ramps, e.g., see the right panel of Fig.~\ref{fig:Map}. We consider only a subset of actual ramps for the sake of illustration. 

Every combination of on- and off-ramp is represented by three cells: the two ramps, and a section of the main line between the ramps. Consistent with the real network, the off-ramp is always located before the on-ramp. For example, the on- and off-ramps of I-405 at Carson St are represented by the cells 57-84-58 in the northbound direction, and 59-85-60 in the southbound direction. In addition to actual on- and off-ramps, sections of the main lines that arrive to or depart from the transportation network under study are also considered as on- and off-ramps, respectively. For example, cells 1 and 12, though being part of the mainline of I-110, will be considered as on- and off-ramps, respectively. Finally, interconnections between main lines are represented by a set of merge and diverge nodes. For example, the interconnection between state route 91 and I-110 is represented by the tail and head nodes of cell 69. Overall, the network under study consists of 91 cells, including 22 on-ramps and 22 off-ramps.

\begin{figure}[htb!]
\begin{center}
\begin{tabular}{lr}
\includegraphics[height=9.5cm]{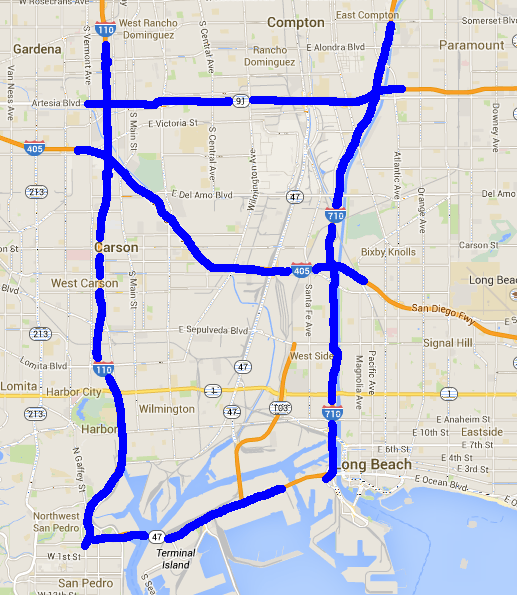}&
\includegraphics[height=9.5cm]{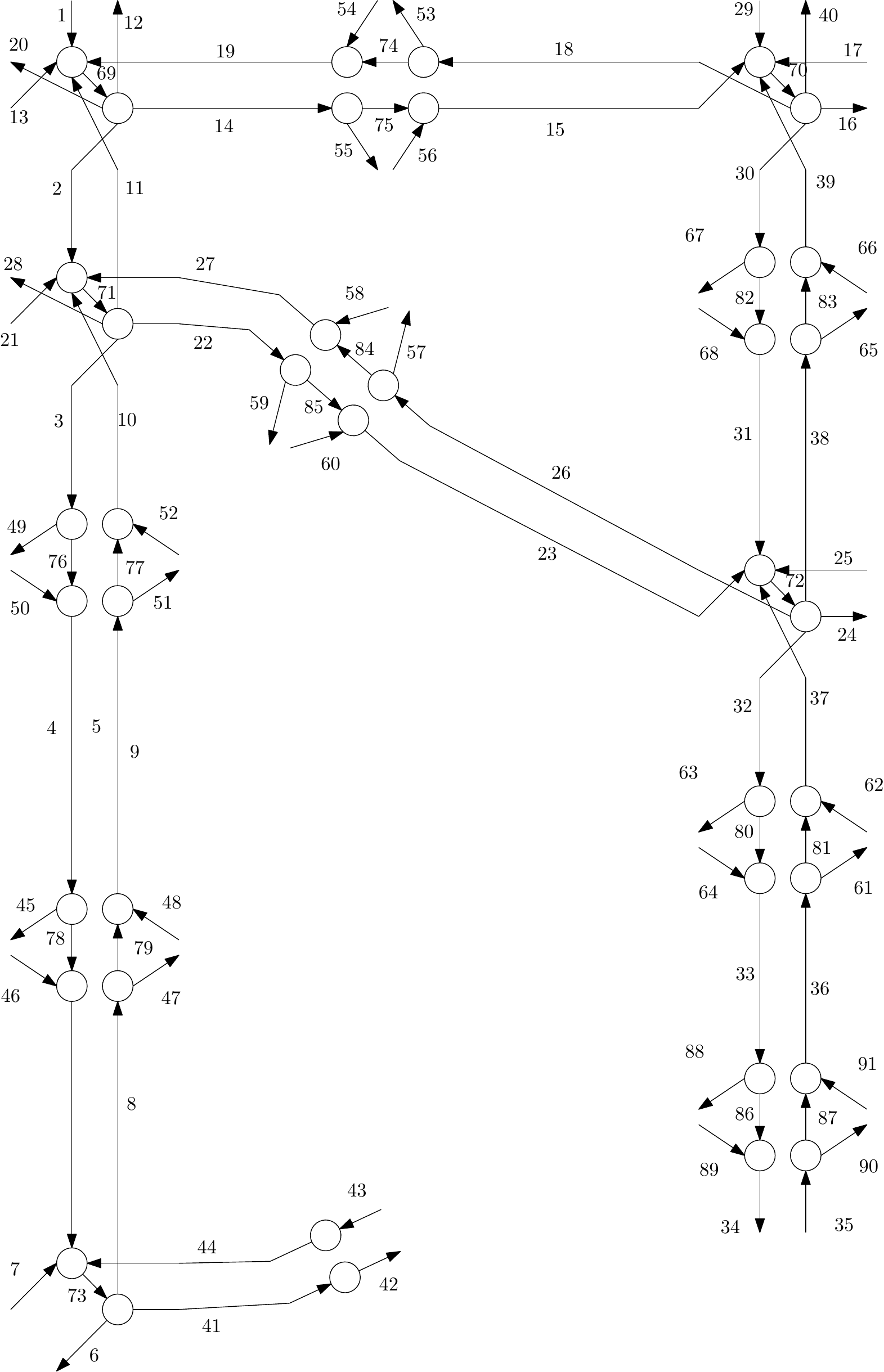}
\end{tabular}
\end{center}
\caption{Left: map of the area of interest in the southern Los Angeles region, with the portions of state routes and interstate freeways that are used for our simulation study, shown in blue. Right: the corresponding directed graph representation.}
\label{fig:Map}
\end{figure}

\begin{remark}
As mentioned in Remark~\ref{rem:densityVSvolume}, in this section we interpret the state $\rho_e(t)$ of the cell $e$ as the volume of vehicles on $e$ at time $t$, rather then its density. As such, demand and supply functions need slight modifications. In particular, the network parameters are selected as follows. For every cell, we use time-invariant linear demand functions $d_e(\rho_e) = \frac{v_e}{L_e}\rho_e$ and affine supply functions $s_e(\rho_e) = \frac{w_e}{L_e}(\rhomax_e-\rho_e)$, where $L_e$ is the length, $v_e$ is the free-flow speed, $w_e$ is the wave-speed, and $\rhomax_e$ is the jam volume of cell $e$. The values of these parameters are adapted from the PeMS website~\cite{pems}, and summarized in Table~\ref{table:values}.
\end{remark}

\begin{table}
\begin{center}
	\begin{tabular}{c|c|c|c|c}
	Type of cell 											& $L$	 	& $v$ 			&	$w$			& $\rhomax/L$	\\
	\hline
	Main line								&	2 mi			&	65 mph	 	& 13 mph	 	& 200		veh/mi					\\
	Intersections of main lines 			& 0.2 mi 			&	65 mph	 	& 13 mph	 	& 500		veh/mi					\\
	Segments between ramps on main lines 	& 0.5 mi			&	65 mph	 	& 13 mph	 	& 200		veh/mi					\\
	On- and off-ramps 						& 0.5 mi			&	25 mph	 	& 13 mph	 	& $+\infty$/200 veh/mi		\\		
	\end{tabular}
\end{center}
\caption{Values of lengths, free-flow speed, wave speed and jam density for different types of cells used in the simulation.}
\label{table:values}
\end{table}

We let the inflow on actual on-ramps be $\lambda = 2$ vehicles per minute, and that on-ramps corresponding to the main lines entering from the external world be $\lambda = 20$ vehicles per minute. The turning preference matrix is also time-invariant, and 
for every node with multiple outgoing cells, it is chosen so that  $R_{ij}=0.1$ if $j$ is an actual off-ramp, and $1-R_{ij}$ is split uniformly between the remaining outgoing cells. For example, $R_{69,2} = R_{69,12} = R_{69,14} = R_{69,20} = 0.25$, whereas $R_{18,53} = 0.1$ and $R_{18,74} = 0.9$.

Finally, the continuous-time system \eqref{eq:systfirst} is discretized with a first-order Euler method with step size $\Delta t = 10$ seconds, which is sufficiently small to satisfy the Courant-Friedrichs-Lewy condition $\max_e\frac{v_e\Delta t}{L_e}  \approx 0.9 \leq 1$ \cite{WorkAMRX10, LeVeque:92}.

\subsection{Stability of free-flow equilibrium and response to traffic incidents}
\label{subsect:stabilityMixed}

We first report results to illustrate the stability of dynamical transportation networks under non-monotone policies. 
In particular, we compare the performance under the mixture model from Example~\ref{example:mixture} for (i) $\theta=0$, which corresponds to the non-FIFO policy in Example~\ref{example:policiesNonFIFO}; (ii) $\theta=1$, which corresponds to the FIFO policy in Example~\ref{example:policiesFIFO}; and (iii) $\theta=0.8$. We run two sets of simulations: in the first set, we investigate the effect of $\theta$ on stability of equilibrium $\rho^*$, as suggested by Remark~\ref{remark:nonmonotone}, and in the second set, we investigate the ability of the network to respond to traffic incidents for different $\theta$.

\subsubsection{Stability of free-flow equilibrium under non-monotone policies}
\label{subsubsec:stability}

One can show that, independent of $\theta \in [0,1]$, there exists a free-flow equilibrium $\rho^*$ with $\rho^*_e = d_e^{-1}(f_e^*)$ with $f^* = (I-R^T)^{-1}\lambda$.

In order to investigate stability of $\rho^*$, we consider trajectories starting from $\rho(0)=\zerobf$, $\rho(0)=3 \rho^*$, and $\rho(0)= \rhomax$, except $\rho_e(0) = 100$ vehicles if $e$ is an on-ramp. The results are shown in Fig.~\ref{fig:ComparisonnonFIFODiffInCondFreeFlowSmallInCond} and Fig.~\ref{fig:ComparisonL1NormsDifferencesDifferentInitCondFreeflow}. In Fig.~\ref{fig:ComparisonnonFIFODiffInCondFreeFlowSmallInCond}, we show the evolution of the volumes of vehicles, $\rho_e$, only on cells $1$, $27$, $54$ and $84$ (for the sake of brevity in presentation) for $\theta=0$, i.e., under the non-FIFO policy in Example~\ref{example:policiesNonFIFO}. These results illustrate the global asymptotic stability result of Theorem~\ref{theo:freeflow}. 

In Fig.~\ref{fig:ComparisonL1NormsDifferencesDifferentInitCondFreeflow}, we show the evolution of $\|\rho(t)-\rho^*\|_1$ under non-FIFO and FIFO policies, i.e., for $\theta = 0$ and $\theta = 1$, comparing evolutions with initial conditions $\rho(0)=\zerobf$ and $\rho(0)=3 \rho^*$ (left panel), and $\rho(0)=\zerobf$ and $\rho(0)= \rhomax$ (right panel). These results illustrate Remark~\ref{remark:nonmonotone}. Indeed, while in the former case the network under both policies converges to the free-flow equilibrium, thus illustrating global asymptotic stability under non-monotone policies, in the latter the large initial condition prevents the non-monotone FIFO policy to steer the network to equilibrium. In fact, on the contrary, the trajectory of the system under FIFO policy grows unbounded, thus numerically showing that non-monotone policies cannot guarantee global asymptotic stability of the free-flow equilibrium for general networks, as already discussed in Example~\ref{example:counterexampleGASFIFO}.

\begin{figure}[htb!]
\centering 
\includegraphics[height=7cm]{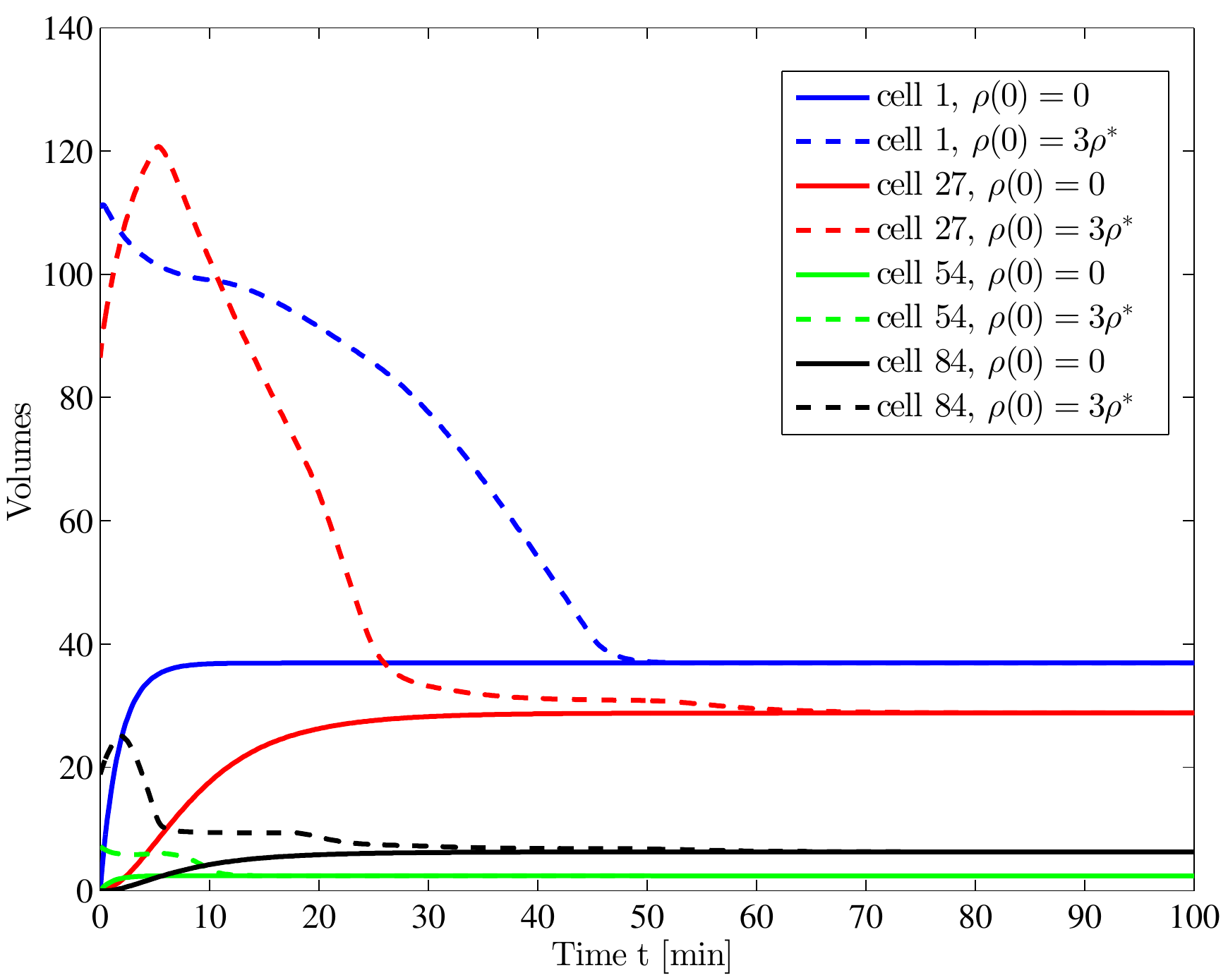}
\caption{Evolution of the trajectories of the dynamical transportation network starting from $\rho(0)=\zerobf$ (solid), and $\rho(0)=3 \rho^*$ (dashed), where $\rho^*$ is the free-flow equilibrium, under the non-FIFO policy in Example~\ref{example:policiesNonFIFO}. For brevity, we show evolution of volumes of vehicles only on cell $1$, $27$, $54$ and $84$. 
}

\label{fig:ComparisonnonFIFODiffInCondFreeFlowSmallInCond}
\end{figure}

\begin{figure}[htb!]
\centering 
\begin{tabular}{cc}
\includegraphics[height=6cm]{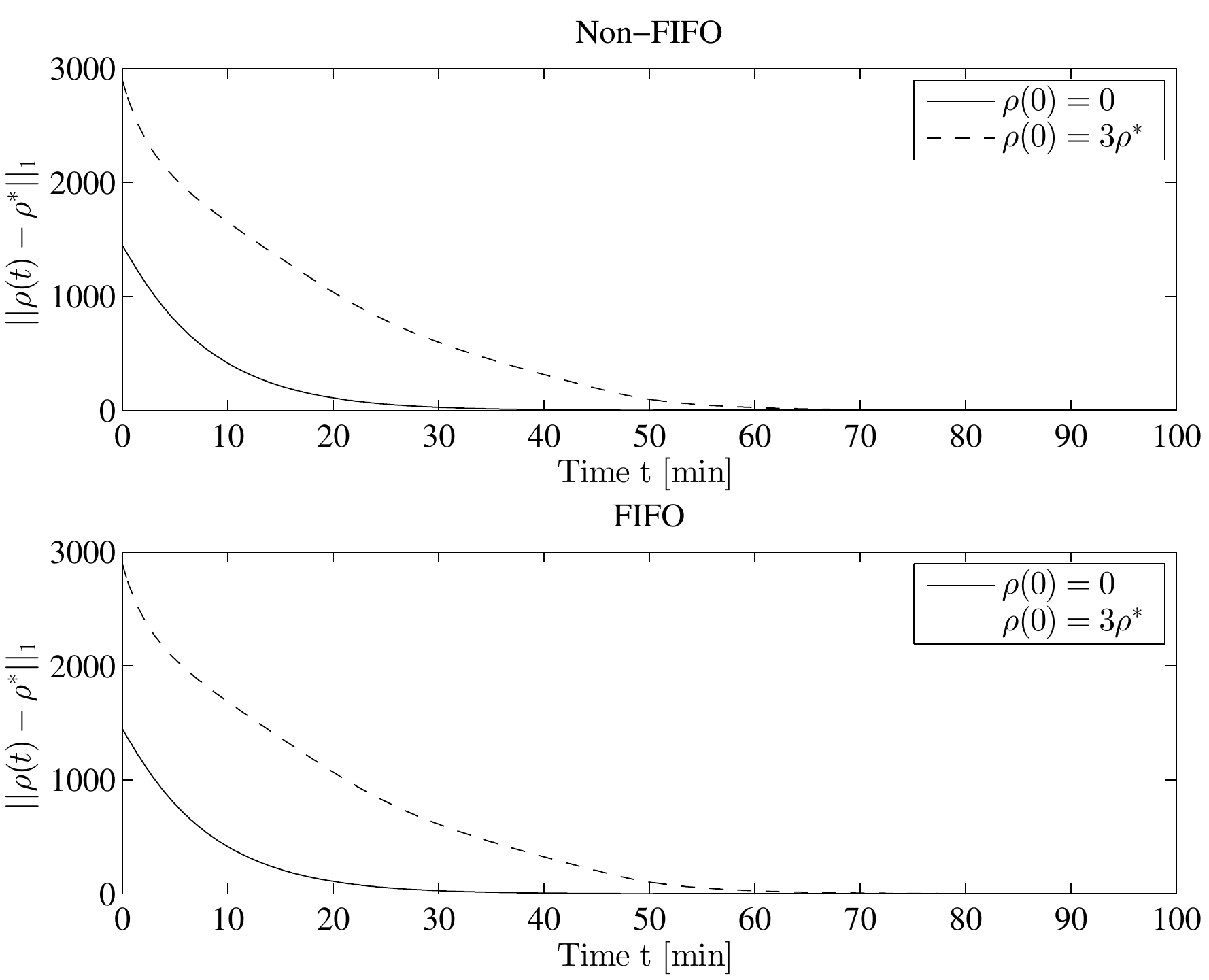}&
\includegraphics[height=6cm]{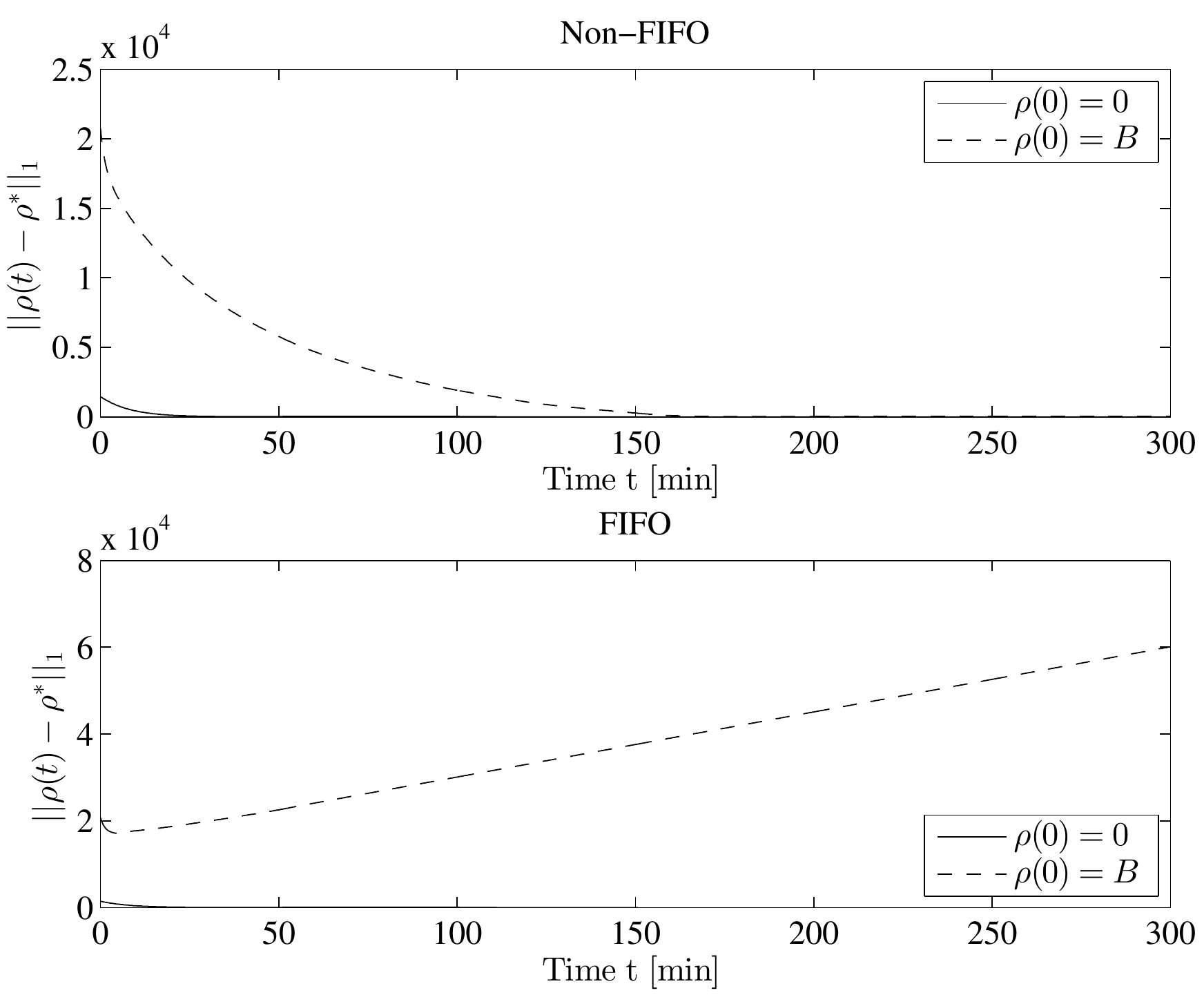}
\end{tabular}
\caption{Evolution of $\|\rho(t)-\rho^*\|_1$ for the dynamical transportation network under $\theta=0$ (top) and $\theta=1$ (bottom), for initial condition $\rho(0) = 3\rho^*$ (left panel) and $\rho(0) = \rhomax$, except $\rho_e(0) = 100$ vehicles if $e$ is an on-ramp (right panel).
}
\label{fig:ComparisonL1NormsDifferencesDifferentInitCondFreeflow}
\end{figure}

\subsubsection{Response to traffic incidents}
\label{subsubsec:response}
We considered a congestion scenario in which a bottleneck is present since $t = 0$ on cell $27$, modeled by reduction of the free-flow speed from $v_{27} = 65$ mph to $v_{27} = 4$ mph. As a consequence, the flow capacity on cell $27$ drops to $C_{27} = 10$ vehicles per minute. We plot the resulting evolution of $\sum_{e\in\E}\rho_e(t)$, i.e., the total number of vehicles in the system for $\theta=0$, $\theta=1$ and $\theta=0.8$ for initial condition $\rho(0)=\zerobf$, in Fig.~\ref{fig:ComparisonMixing09Sum}. In Fig.~\ref{fig:ComparisonL1NormsDifferencesDifferentInitCond}, we plot the evolution of $\|\rho_1(t)-\rho_2(t)\|_1$ where $\rho_1(t)$ and $\rho_2(t)$ are the evolutions for initial conditions $\rho_1(0) = \zerobf$ and $\rho_2(0) = \rhomax$, except $\rho_{2,e}(0) = 100$ vehicles if $e$ is an on-ramp.

\begin{figure}[htb!]
\centering 
\includegraphics[height=7cm]{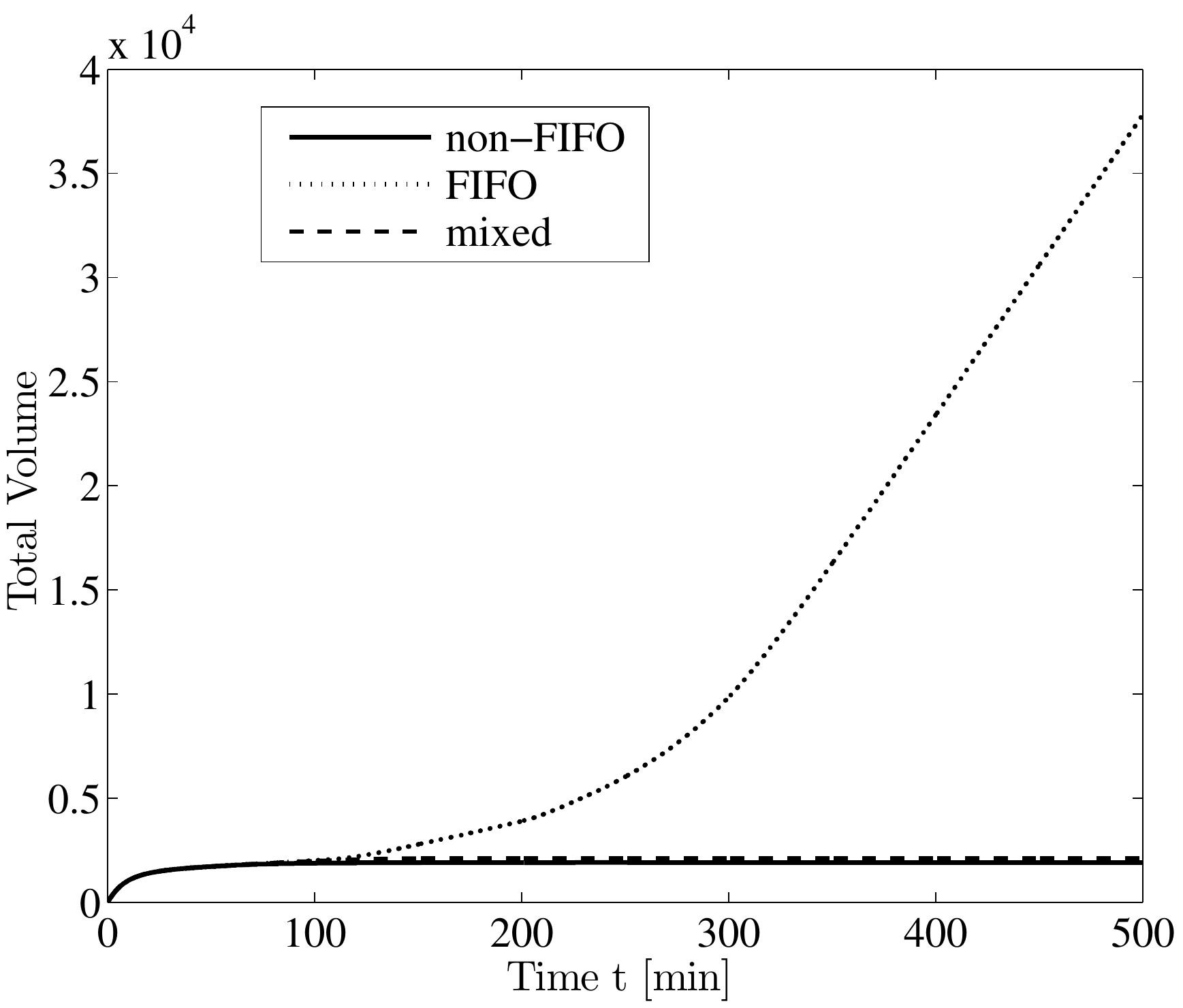}
\caption{Evolution of the total number of vehicles, $\sum_{e\in\E}\rho_e(t)$ after the introduction of bottleneck on cell $27$ at $t=0$, causing its capacity to drop to $C_{27} = 10$ vehicles per minute, for $\theta=0$ (solid), $\theta=1$ (dotted) and $\theta=0.8$ (dashed). All the trajectories start from $\rho(0)=\zerobf$.}
\label{fig:ComparisonMixing09Sum}
\end{figure}

\begin{figure}[htb!]
\centering 
\includegraphics[height=7cm]{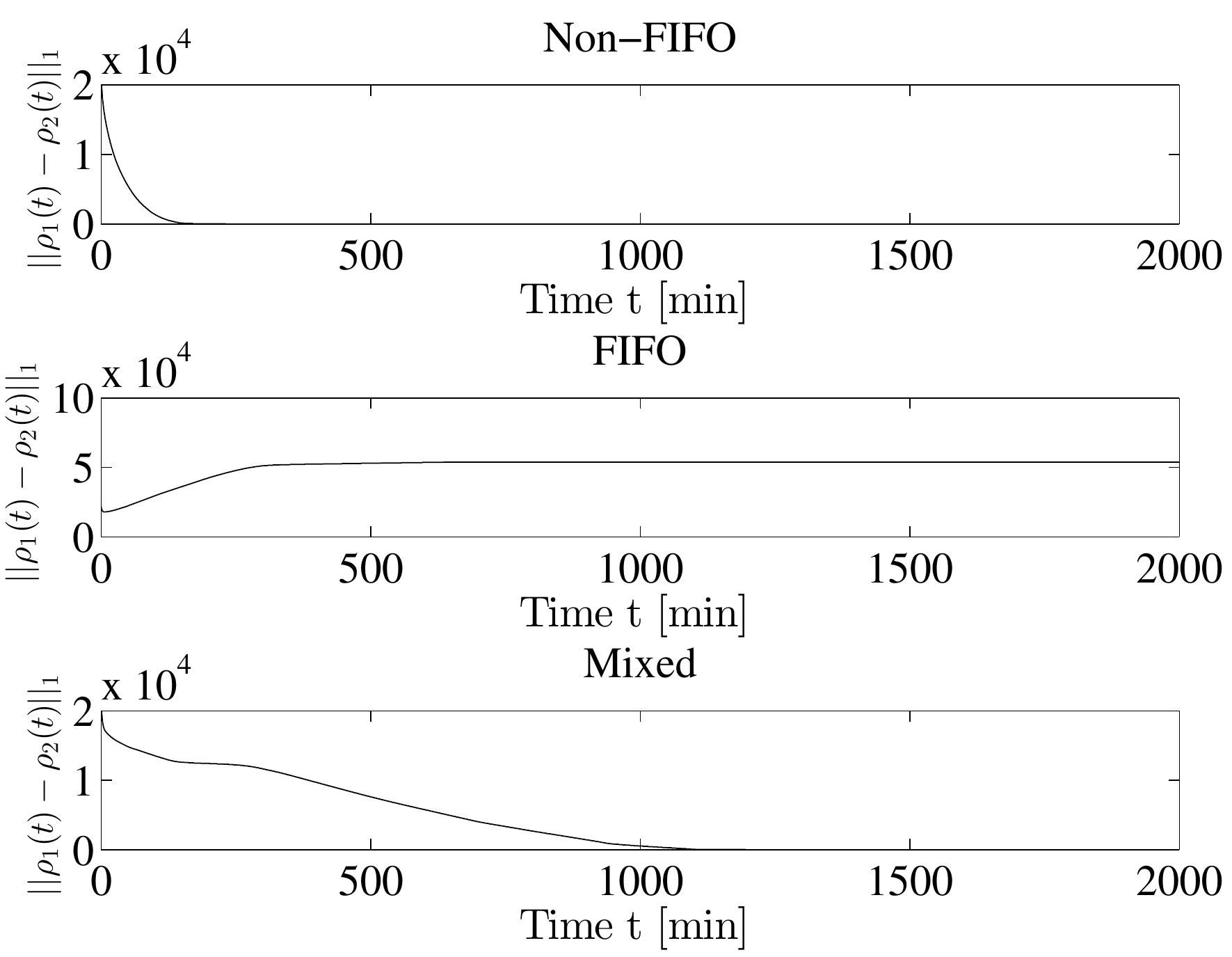}
\caption{Evolution of $\|\rho_1(t)-\rho_2(t)\|_1$ for the dynamical transportation network under $\theta=0$ (top), $\theta = 1$ (middle), and $\theta=0.8$ (bottom), where $\rho_1(t)$ and $\rho_2(t)$ are the evolutions of the volumes of vehicles with initial conditions $\rho_1(0) = \zerobf$ and $\rho_2(0) = \rhomax$, except $\rho_{2,e}(0) = 100$ vehicles if $e$ is an on-ramp, respectively.
}
\label{fig:ComparisonL1NormsDifferencesDifferentInitCond}
\end{figure}

First of all, under non-FIFO policy, it can be seen that the bottleneck on cell $27$ causes the cells $27$ and $84$ to be congested at equilibrium, namely, their inflows to be bounded by their supplies. Congestion is however limited to these cells and does not spread in the rest of the network. This phenomenon is due to the fact that under such a policy vehicles that cannot proceed along the preferred path are rerouted towards off-ramps or other freeways. In this case, vehicles on cell $26$ are rerouted to off-ramp $57$ instead of entering the congested cell $84$. The network therefore does not become unstable due to the bottleneck, and the system reaches a new non free-flow equilibrium. Since it can be shown by inspection that the dual graph is rooted, such a non free-flow equilibrium is globally asymptotically stable by Proposition~\ref{prop:GAS}.

Conversely, under FIFO policy vehicles are constrained to follow the turning preference matrix and to form queues when they encounter congestion. Since cell $27$ cannot accept the flow prescribed by the turning preference matrix, which is given by $f^*_{27} = [(I-R^T)^{-1}\lambda]_{27} \approx 15 > 10 = \fmax_{27}$ vehicles per minute, congestion spills back upstream until it reaches the on-ramps, on which vehicles queue up unbounded, see Figure~\ref{fig:ComparisonMixing09Sum}. 

Finally, although most of the vehicles tend to follow the fixed turning preference matrix, the system is stable under mixed policy too, and moreover both evolutions starting with empty and congested networks reach the same equilibrium (lower panel, Figure~\ref{fig:ComparisonL1NormsDifferencesDifferentInitCond}). The consequence of the reduced flexibility with respect to the pure non-FIFO policy can be finally observed in terms of total volume of vehicles in the network at equilibrium, which is higher in case of mixed policies, see Figure~\ref{fig:ComparisonMixing09Sum}.

\subsection{Equilibrium selection}

We now report simulation results to illustrate the findings from Section~\ref{subsec:equilibriumcontrol}. We consider the same setup as in Section~\ref{subsubsec:response} involving introduction of bottleneck on cell $27$ at $t=0$. In the present case, we consider the possibility of introducing controls in response to this incident. We consider controls of type \textbf{(I)}. The cost function is assumed to be $\Psi(\rho) = \sum_{e\in\E}\rho_e$, i.e., the total number of vehicles at the new equilibrium. The evolution of the system under controlled and uncontrolled cases is shown in Fig.~\ref{fig:ImprovementEquilibriumLongBeach}. 

By inspecting the optimum control signals, it can be seen that $R_{72, 36} = 0$, namely, the optimal control prevents vehicles from taking the $26$-$84$-$27$ branch of I-405, consequently avoiding the bottleneck on cell $27$. Fig.~\ref{fig:ImprovementEquilibriumLongBeach} shows improvements in network performance at the new equilibrium after the traffic incident, under the proposed control in comparison to the uncontrolled case. In particular, the controlled equilibrium does not exhibit high congested volumes of vehicles and its asymptotic total volume is reduced by a factor of around $4$ with respect to the uncontrolled equilibrium.

\begin{figure}[htb!]
\centering 
\begin{tabular}{lr}
\includegraphics[height=5cm]{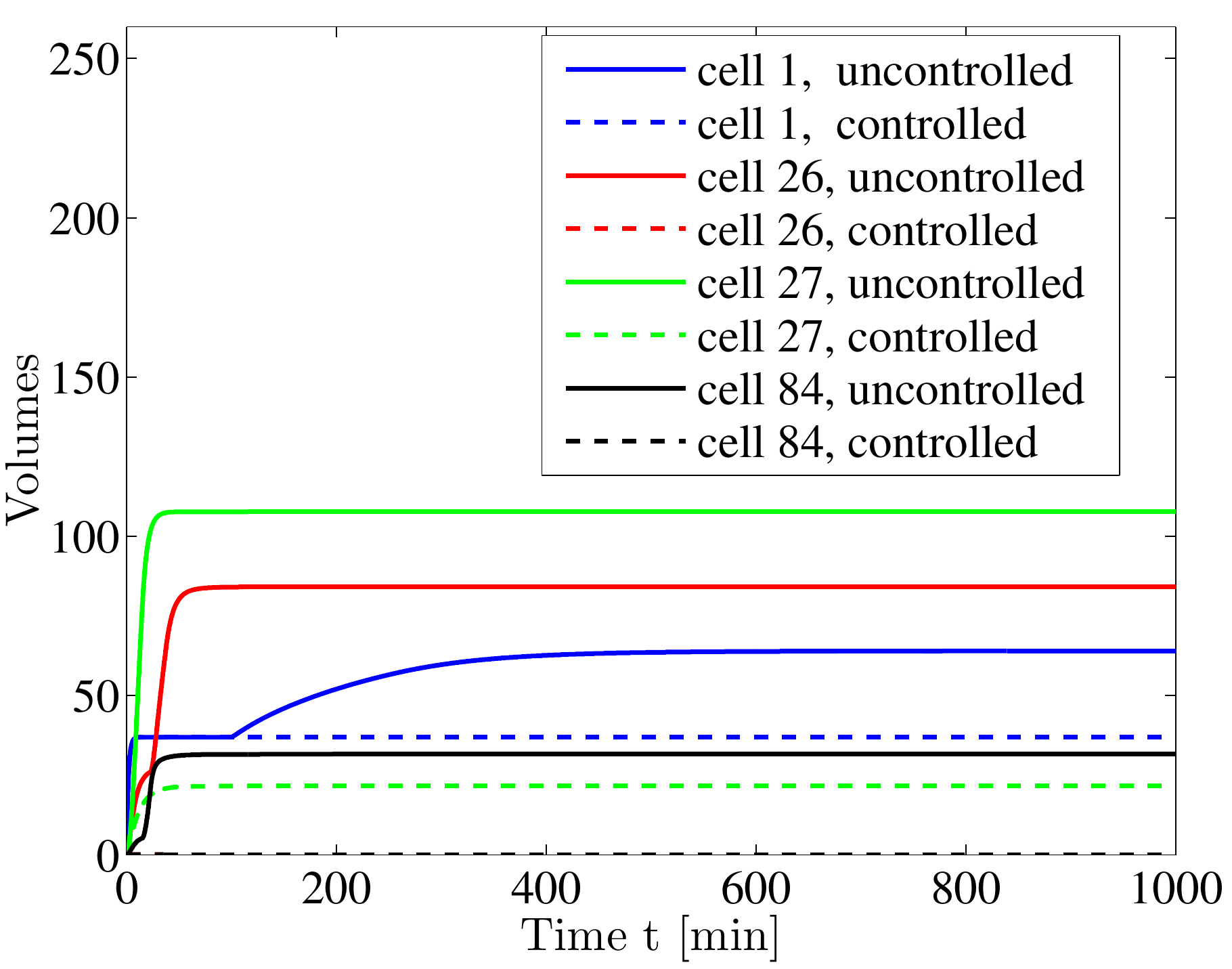}&
\includegraphics[height=5cm]{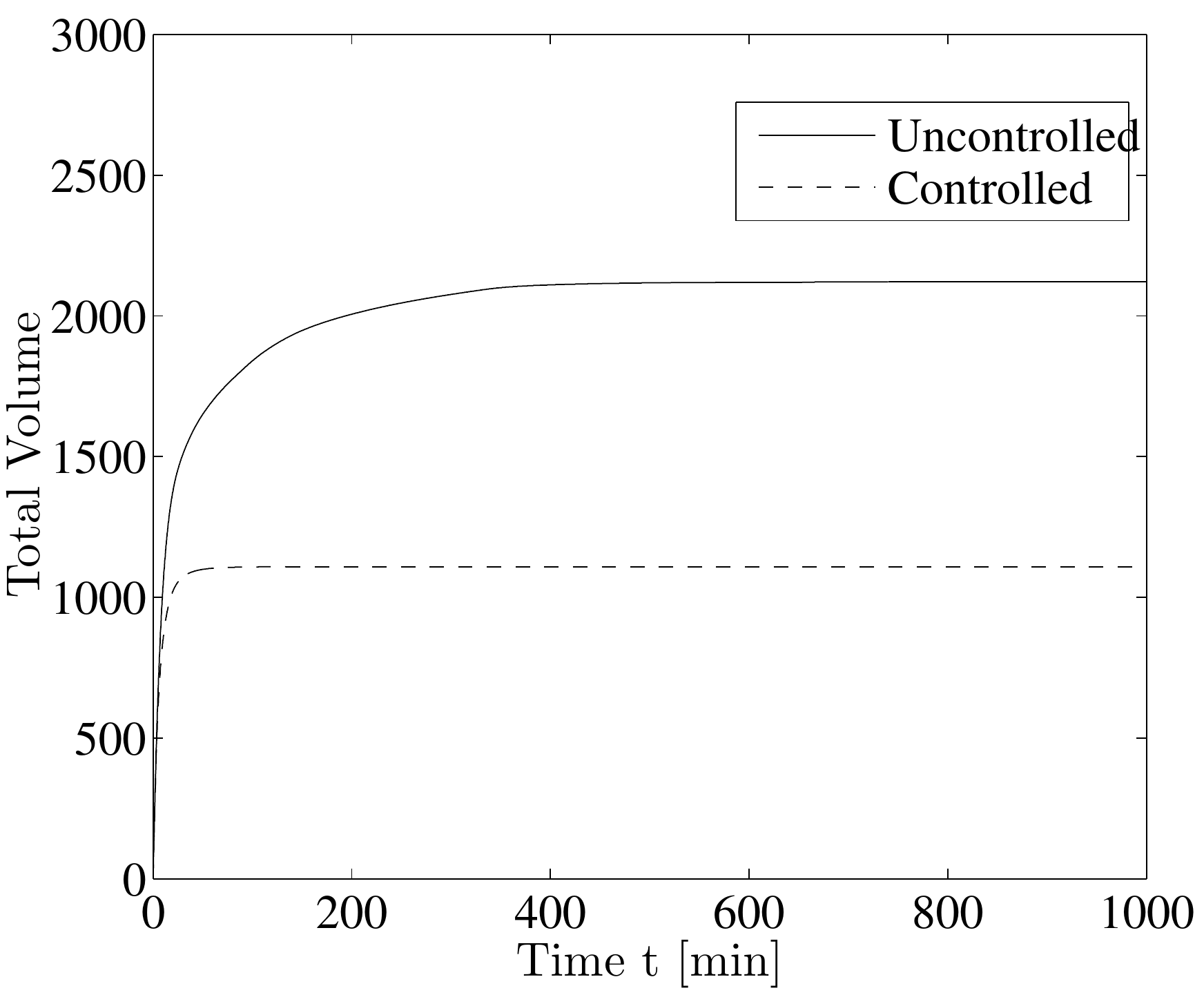}
\end{tabular}
\caption{Simulation with bottleneck on cell $27$ and corresponding capacity drop to $C_{27} = 8$. Left panel: trajectory of the uncontrolled (solid line) and uncontrolled (dashed line) volumes of vehicles on cells $1$, $26$, $27$, $84$ with initial condition $\rho(0) = 0$. Right panel: trajectory of the total volume of vehicles in the network.}
\label{fig:ImprovementEquilibriumLongBeach}
\end{figure}

\subsection{Optimal control}

We now report simulation results to illustrate the findings from Section~\ref{subsec:timevaryingcontrol}. We consider the setup of Section~\ref{subsubsec:stability}, where there is no bottleneck on cell $27$. We consider the evolution of system trajectory starting from initial condition $\rho_e(0)=\rhomax/2$ vehicles for all $e \in \E$, except $\rho_e(0) = 50$ vehicles if $e$ is an on-ramp, and solve the optimal control problem in \eqref{optimizationTimeVarying} for $H=5$ minutes, with the cost function being $\sum_{e\in\E}\int_0^H \rho_e(s) \, ds$.

The solution of this optimization is used to control the system as follows. At $t=0$, the control is computed for the horizon $[0, H]$, and executed over $[0, H]$. At time $H$, the optimization is solved again to compute control over $[H,2H]$. The procedure is repeated over a period of $3$ hours, namely, up to $[35H, 36H]$. Increasing the horizon to, say $H = 30$ minutes, would yield better performance, but it would also increase the computational cost of finding the optimal control.
The evolution of the cost for the trajectory for uncontrolled and controlled systems is shown in Fig.~\ref{fig:ImprovementTrajectoryLongBeach}. As in the case of equilibrium selection, the controlled system performs significantly better than the uncontrolled one.

\begin{figure}
\centering 
\begin{tabular}{lr}
\includegraphics[height=5cm]{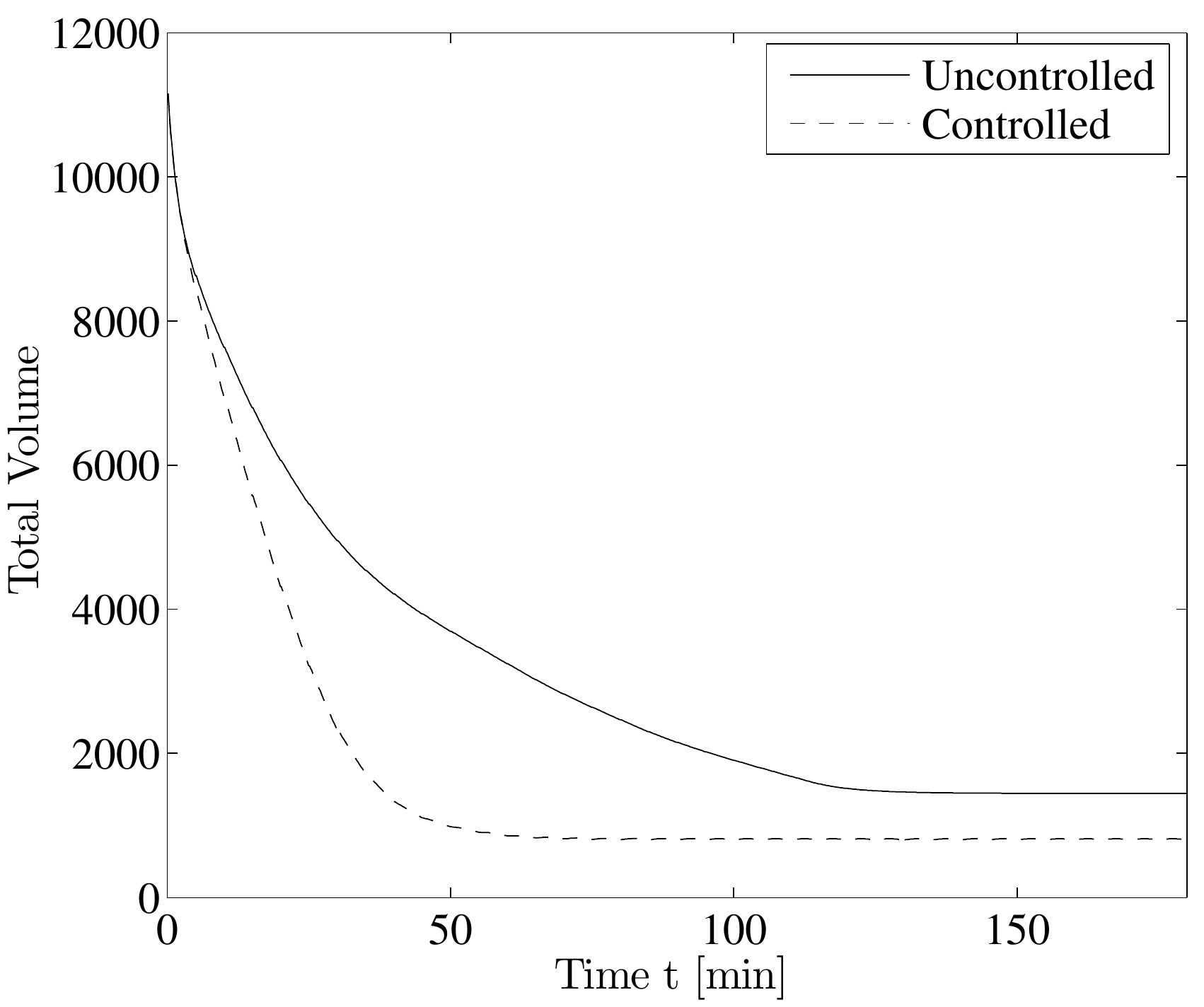}
\end{tabular}
\caption{Optimal trajectory selection. Sum of the trajectories of the uncontrolled (solid line) and uncontrolled (dashed line) systems with initial condition $\rho_e(0)=\rhomax/2$ vehicles for all $e \in \E$, except $\rho_e(0) = 50$ vehicles if $e$ is an on-ramp.}
\label{fig:ImprovementTrajectoryLongBeach}
\end{figure}

\section{Conclusions}
\label{section:conclusions}
We considered dynamical transportation networks, where the dynamics are governed by the demand and supply functions on cells that relate densities and flows, merging and splitting rules at junctions, and inflows at on-ramps. In particular, this framework includes extensions of the classical Cell Transmission Model to arbitrary network topologies. We provide sufficient conditions for stability of equilibria and periodic orbits for such networks, in terms of the connectivity of a state-dependent dual graph. We also formulated an optimal control synthesis problem, and identified sufficient conditions under which this formulation is convex. 

Future research directions include generalization of the results in this paper to non monotone dynamics, as suggested by Remark~\ref{remark:nonmonotone}, robustness analysis along the lines of our previous work~\cite{ComoPartITAC13,ComoPartIITAC13}, and developing scalable implementations of the optimal control solutions.


\appendix

\section{Nonlinear dynamical sytems}
\label{dynamicalSystems}

In this brief appendix, we gather some basic concepts and definitions from nonlinear dynamical systems. We refer to \cite{KhalilBook:02} for a thorough treatment. 
A dynamical system with state $x \in \X$ is a system of the type
	\begin{equation}
	\label{eq:dynSysApp}
		\dot{x} = g(t, x)\,.
	\end{equation}
If $g(t,x) = g(x)$ does not depend on time, the system said to be autonomous and the evolution does not depend on the initial time. Otherwise, the system is said to be non-autonomous. 

An element $x^* \in\X$ is an equilibrium for \eqref{eq:dynSysApp} if $g(x^*,t ) = 0$ for all $t\geq t_0$, so that the solution with initial condition $x^*$ is such that $x(t) = x^*$ for all $t\geq t_0$. A solution $x(t)$ is a nontrivial periodic solution if there exists $T > 0$ such that $x(t+T) = x(t)$ for all $t\geq t_0$. A solution $x(t)$ is converging if $\lim\limits_{t\to+\infty}x(t) \in \X$ exists.

For an autonomous system, an equilibrium $x^*$ is
	\begin{itemize}
		\item stable, if, for each $\varepsilon > 0$, there exists $\delta(\varepsilon)>0$ such that 
			$$
				||x(t_0) - x^*|| < \delta \Rightarrow ||x(t) - x^*||<\varepsilon, \forall t\geq t_0
			$$
		for some norm $||\cdot||$ in $\X$;
		\item locally asymptotically stable, if it is stable and there exists $\delta > 0$ such that
			$$
				||x(t_0) - x^*|| < \delta \Rightarrow \lim_{t\to\infty}x(t) = x^*
			$$	
		\item globally asymptotically stable, if it is stable and, for any $x(t_0)\in\X$
			$$
				\lim_{t\to\infty}x(t) = x^*
			$$	
	\end{itemize}
	
For an autonomous system, a periodic trajectory $x^p(t)$ with $x^p(t+T) = x^p(t)$ for all $t\geq t_0$ is
	\begin{itemize}
		\item stable, if, for each $\varepsilon > 0$, there is $\delta = \delta(\varepsilon)$ such that 
			$$
				||x(t_0) - x^p(t_0)|| < \delta \Rightarrow ||x(t) - x^p(t)||<\varepsilon, \forall t\geq t_0
			$$
		for some norm $||\cdot||$ in $\X$;
		\item locally asymptotically stable, if it is stable and there exists $\delta > 0$ such that
			$$
				||x(t_0) - x^p(t_0)|| < \delta \Rightarrow \lim_{t\to\infty}||x(t) - x^p(t)|| = 0
			$$	
		\item globally asymptotically stable, if for any $x(t_0)\in\X$
			$$
				 \lim_{t\to\infty}||x(t) - x^p(t)|| = 0
			$$	
	\end{itemize}

\section{Technical results}
\label{technicalresults}

\subsection{\texorpdfstring{$\ell_1$}{l1}  contraction principle}

The next result is a simple adaptation of the $\ell_1$ contraction principle for monotone dynamical systems with mass conservation that is proven in \cite{Como.Lovisari.ea:TCNS13}.

\begin{lemma}
\label{lem:l1-contraction} Let $g: \RR_+^m\times \RR_+ \to\RR^m$, $(x,t)\to g(x,t)$ be a Lipschitz map such that
    \begin{equation}
    \label{equation:contractionAssumption1}
        \frac{\partial}{\partial x_j}g_i(x,t) \geq0\,,\qquad \forall\,i\ne j\in\{1,\ldots,m\}, \forall t\geq 0
    \end{equation}
and that
    \begin{equation}
    \label{equation:contractionAssumption2}
            \sum_{1\le i\le m}\frac{\partial}{\partial x_j}g_i(x,t) \leq 0\,,\qquad \forall\,j \in\{1,\ldots,m\}
    \end{equation}
for every $x\in \RR_+^m$, $t\geq0$. Then 
	\begin{equation} 
		\label{ineq:contraction}
		\sum_{1\le i\le m}\sgn{x_i-y_i} \left( g_i(x,t) -g_i(y,t) \right)\leq 0,\mbox{ } \forall x,y \in \RR_+^m, t\geq0\,.
\end{equation}
\end{lemma}

\subsection{Proof of Lemma~\ref{lemma:l1strictlydecreasing}}

Let $\rho^{(1)}(\cdot)$ and $\rho^{(2)}(\cdot)$ be two solutions of \eqref{eq:systfirst} starting at time $t$ with initial conditions $\rho^{(1)}(t) = \rho^{(1)}$ and $\rho^{(2)}(t) = \rho^{(2)}$, and assume $\H(\rho^{(1)},t)$ to be rooted. Let 
	$$
		\rho^c = \rho^{(1)} - \frac{\varepsilon}{||\rho^{(1)} - \rho^{(2)}||_1}(\rho^{(1)} - \rho^{(2)})
	$$
for $\varepsilon>0$ small enough. Then it holds
	\begin{align*}
		||\rho^{(1)} - \rho^c||_1 & = \varepsilon 	\\
		||\rho^{(1)} - \rho^{(2)}||_1 & = ||\rho^{(1)} - \rho^c||_1 + ||\rho^c - \rho^{(2)}||_1
	\end{align*}
Let $\rho^c(\cdot)$ be the solution of the system starting at time $t$ with initial condition $\rho^c$, namely $\rho^c(t) = \rho^c$. Then by the triangle inequality and Theorem~\ref{theo:monotonicity+contraction},
	\begin{align*}
		||\rho^{(1)}(t+h) - \rho^{(2)}(t+h)||_1
			&	\leq ||\rho^{(1)}(t+h) - \rho^c(t+h)||_1 + ||\rho^c(t+h) - \rho^{(2)}(t+h)||_1
				\\
			&	\leq ||\rho^{(1)}(t+h) - \rho^c(t+h)||_1 + ||\rho^c - \rho^{(2)}||_1
				\\
			&	= ||\rho^{(1)}(t+h) - \rho^c(t+h)||_1 + ||\rho^{(1)} - \rho^{(2)}||_1 - ||\rho^{(1)} - \rho^c||_1 
	\end{align*}
so (recall that $\rho^{(1)}(t) = \rho^{(1)}$ and $\rho^c(t) = \rho^c$)
	$$
		||\rho^{(1)}(t+h) - \rho^{(2)}(t+h)||_1 - ||\rho^{(1)}(t) - \rho^{(2)}(t)||_1
		\leq
		||\rho^{(1)}(t+h) - \rho^c(t+h)||_1 - ||\rho^{(1)}(t) - \rho^c(t)||_1	
	$$
hence $\frac{\de}{\de t}||\rho^{(1)}(t)-\rho^{(2)}(t)||_1<0$ if we prove that $\frac{\de}{\de t}||\rho^{(1)}(t)-\rho^c(t)||_1<0$, namely, that
	$$
		\sum_{i}\sgn{\rho^{(1)}_{i} - \rho_i^c}(g_i(\rho^{(1)},t) - g_i(\rho^c,t)) < 0
	\,.
	$$
To this aim, for $\A\subseteq\E$, put $\A^c:=\E\setminus \A$, and $g_{\A}(z):=\sum_{i\in\A}g_i(z)$. Let $ \I = \{i: \rho^{(1)}_i > \rho^c_i\}$, $\J = \{i: \rho^{(1)}_{i} < \rho^c_i\}$. Let $\xi\in\S$ be such that $\xi_i = \rho^{(1)}_{i}$ for $i \in \I$ and $\xi_i = \rho_i^c$ for $i \in \I^c$. Consider the
segments $\gamma_{\I}$ from $\rho^c$ to $\xi$ and $\gamma_{\J}$ from $\rho^{(1)}$ to $\xi$. For $\A\subseteq\E$,
and $\B\in\{\I,\J\}$, define the path integral
    $$
        \Gamma^{\A}_{\B}:=\int_{\gamma_{\B}}\nabla g_{\A}(z)\cdot \de z\,.
    $$
Then
    \begin{align}
      	g_{\I}(\rho^{(1)})-g_{\I}(\rho^c)
        	&	=\ds  \Gamma^{\I}_{\I}-\Gamma^{\I}_{\J} = -\Gamma^{\I^c}_{\I} + \Gamma_{\I}^\E - \Gamma^{\I}_{\J} \label{ineq1}
        		\\
       g_{\J}(\rho^{(1)})-g_{\J}(\rho^c) 
       		&	= \Gamma^{\J}_{\I}-\Gamma^{\J}_{\J} = \Gamma^{\J}_{\I}+\Gamma^{\J^c}_{\J} - \Gamma_{\J}^\E \,.\label{ineq2}
    \end{align}

It thus holds true
	\begin{align*}
		\sum_{i}\sgn{\rho^{(1)}_{i} - \rho_i^c}(g_i(\rho^{(1)},t) - g_i(\rho^c,t))
			&	=	g_{\I}(\rho^{(1)})-g_{\I}(\rho^c)-g_{\J}(\rho^{(1)})+g_{\J}(\rho^c)	\\
			& = -\Gamma^{\I^c}_{\I} - \Gamma^{\I}_{\J}
			 -\Gamma^{\J}_{\I}-\Gamma^{\J^c}_{\J} + \Gamma_{\I}^\E  + \Gamma_{\J}^\E
	\end{align*}
where notice that by monotonicity $\Gamma^{\I^c}_{\I}\geq 0$, $\Gamma^{\I}_{\J}\geq 0$, $\Gamma^{\J}_{\I} \geq0$ and $\Gamma^{\J^c}_{\J}\geq0$, and that since, for any $\rho\in\S$, $\sum_{i\in\E}g_i(\rho,t) = \sum_{i\in\R}\lambda_i - \sum_{i\in\R^o}d_i(\rho_i)$, then $\Gamma_{\I}^\E\leq 0$ and $\Gamma_{\J}^\E\leq0$.

Let $\varepsilon>0$ be small enough so that $\H(\rho,t) = \H(\rho^{(1)},t)$ for all $\rho\in\gamma_{\I}\cup\gamma_{\J}$. Moreover, let $\E$ be partitioned as $\E = \E_1\cup\E_2\cup\dots\E_r$, and the sets $\E_k$ be defined iteratively as follows: $\E_1 = \R^o$, the set of offramps. Let $\E_1, \ldots, \E_k$ be given. $\E_{k+1}$ is the set of all $j \in \E$ such that there exists $i\in\E_k$ and $(j,i)$ is an edge of $\H(\rho^{(1)},t)$, but $(j,e)$ is not an edge of $\H(\rho^{(1)},t)$ for all $e\in\E_l$, $l<k$. Since $\H(\rho^{(1)},t)$ is rooted, $\E_1\cup\E_2\cup\dots\E_r$ is indeed a partition of $\E$. Then
	\begin{itemize}
		\item if $\rho^{(1)}_{i} \neq \rho_i^c$ for some $i\in\E_1$, then at least one among $\Gamma_{\I}^\E$ and $\Gamma_{\J}^\E$ is strictly negative;
		\item assume that $\rho^{(1)}_{i}  = \rho_i^c$ for all $i\in\E_1\cup\dots\cup\E_{k-1}$ and $\rho^{(1)}_{i} \neq\rho^c_i$ for some $i\in\E_k$. Then, by the rooted assumption, at least one among $\Gamma^{\I^c}_{\I}$ and $\Gamma^{\J^c}_{\J}$ is strictly positive.
	\end{itemize}
In both cases, $\sum_{i}\sgn{\rho^{(1)}_{i}  - \rho_i^c}(g_i(\rho^{(1)},t) - g_i(\rho^c,t))< 0$ as required.


\end{document}